\documentclass[11pt,a4paper]{article}
\usepackage[left=2cm, top=3cm,bottom=3cm,right=2cm]{geometry}
\usepackage{mathtools,amssymb,amsthm,mathrsfs,calc,graphicx,xcolor,dsfont,tikz,pgfplots,bm,url,tabularx}
\usepackage[british]{babel}
\usepackage{amsfonts}              
\usepackage[T1,T2A]{fontenc}
\usepackage{enumitem}
\usepackage{tikz-cd}
		\usetikzlibrary{calc} 
		\usetikzlibrary{intersections}
\usepackage[labelfont=bf]{caption}
\usepackage[font=small]{subcaption}

\usepackage[
  bookmarks=true,
  bookmarksnumbered=true,
  bookmarksopen=true,
  unicode=true,
  pdftitle={Random hyperbolic polyhedra in horoballs},
  pdfauthor={Florian Besau, Anna Gusakova, Christoph Thäle},
  pdfsubject={MSC 2020: Primary 60D05; Secondary 52A22, 52A55, 60G55},
  pdfcreator={PDFLaTeX},
  pdfkeywords={
    geodesic convex hull,
    horoball,
    hyperbolic stochastic geometry,
    Poisson--Delaunay tessellation,
    Poisson--Laguerre tessellation,
    random hyperbolic polyhedron
  },
  colorlinks=true,
  linkcolor=black,
  citecolor=black,
  filecolor=black,
  urlcolor=black
]{hyperref}

\usepackage{cleveref}

\theoremstyle{plain}
\newtheorem{theorem}{Theorem}

\newtheorem{corollary}[theorem]{Corollary}
\newtheorem{proposition}[theorem]{Proposition}

\theoremstyle{definition}

\theoremstyle{remark}
\newtheorem{remark}[theorem]{Remark}

\newcommand{\dint}{\textup{d}}

\newcommand{\skel}{\mathop{\mathrm{skel}}\nolimits}
\newcommand{\vol}{\mathop{\mathrm{vol}}\nolimits}

\newcommand{\bx}{\boldsymbol{x}}
\newcommand{\by}{\boldsymbol{y}}
\newcommand{\bz}{\boldsymbol{z}}
\newcommand{\bu}{\boldsymbol{u}}
\newcommand{\bv}{\boldsymbol{v}}
\newcommand{\bw}{\boldsymbol{w}}

\def\BB{\mathbb{B}}
\def\CC{\mathbb{C}}

\def\EE{\mathbb{E}}

\def\HH{\mathbb{H}}

\def\NN{\mathbb{N}}

\def\PP{\mathbb{P}}

\def\RR{\mathbb{R}}
\def\SS{\mathbb{S}}

\def\bp{\mathbf{p}}

\def\bX{\mathbf{X}}

\def\cA{\mathcal{A}}
\def\cB{\mathcal{B}}
\def\cC{\mathcal{C}}
\def\cD{\mathcal{D}}

\def\cH{\mathcal{H}}

\def\cL{\mathcal{L}}

\def\cS{\mathcal{S}}
\def\cT{\mathcal{T}}

\newcommand{\conv}{\mathop{\mathrm{conv}}\nolimits}

\newcommand{\inter}{\operatorname{int}}

\newcommand{\russianL}{\text{\fontencoding{T2A}\selectfont\CYRL}}

\makeatletter
\let\@fnsymbol\@alph
\makeatother

\pgfplotsset{compat=1.18}

\begin{document}

\title{\bfseries Random hyperbolic polyhedra in horoballs}

\author{Florian Besau\footnotemark[1]\;\; --\; Anna Gusakova\footnotemark[2]\;\; --\; Christoph Th\"ale\footnotemark[3]}

\date{}
\renewcommand{\thefootnote}{\fnsymbol{footnote}}
\footnotetext[1]{TU Wien, Austria. Email: florian.besau@tuwien.ac.at}
\footnotetext[2]{University of Münster, Germany. Email: gusakova@uni-muenster.de}
\footnotetext[3]{Ruhr University Bochum, Germany. Email: christoph.thaele@rub.de}

\maketitle

\begin{abstract}
	\noindent
	We study the geodesic convex hull of a stationary Poisson point process restricted to a horoball in $d$-dimensional hyperbolic space. The resulting random set is an unbounded hyperbolic polyhedron with a distinguished ideal direction. Projecting its boundary facets to the bounding horosphere yields a stationary Euclidean tessellation of $\mathbb{R}^{d-1}$, which we identify as a dual Poisson--Laguerre tessellation with an explicit height density. We derive an exact formula for its cell intensity and, in the critical regime where the intensity of the Poisson point process is matched with the height of the horoball, we prove local convergence of the projected tessellation to the classical Poisson--Delaunay tessellation in $\mathbb{R}^{d-1}$. As consequences, the typical cell converges in distribution and the intensities of all $k$-dimensional faces converge to their Poisson--Delaunay counterparts. We also study a localised volume functional of the hyperbolic convex hull, determine its limiting expectation in the same regime, and use an Efron-type identity to obtain a second-order asymptotic expansion for the vertex intensity. In addition, we derive an exact formula for the expected localised surface area and determine its critical asymptotics. In dimensions $d\ge3$ the expected unnormalised local surface area converges to a finite limit, whereas in dimension $d=2$ it exhibits logarithmic growth.

\end{abstract}

{
\footnotesize
\tableofcontents
}

\section{Introduction and overview of main results}

\subsection{General introduction and motivation}

Random polytopes are one of the classical objects of convex, integral and stochastic
geometry. Starting with the seminal work of R\'enyi and Sulanke
\cite{RenyiSulanke1963,RenyiSulanke1964}, an extensive literature has
developed on the asymptotics of convex hulls of random points in Euclidean convex bodies,
see, for example, the surveys \cite{Barany2008, Hug2013, Reitzner2010} and Chapter 8 in the monograph
\cite{SW}. In its most classical form, one fixes a convex body
$K\subset\RR^d$, chooses independent and uniformly distributed random points $\bX_1,\ldots,\bX_n$ in $K$, and studies the random convex hull $K_n$ of $\bX_1,\ldots,\bX_n$.
A closely related Poissonised version is obtained by taking the convex hull
of the points of a homogeneous Poisson point process inside $K$, see \cite{BaranyReitzner2010}.
Typical questions
concern the asymptotic behaviour of geometric and combinatorial
characteristics of this random polytope, such as its volume, intrinsic volumes, number of
faces, surface area, or the geometry of its boundary, as the number of points
or the intensity of the underlying point process tends to infinity.

In recent years, similar questions have also been investigated in
non-Euclidean geometries. For random polytopes generated by uniformly
distributed points in a hyperbolic convex body, the asymptotic behaviour of
the expected missed volume was determined in \cite{BLW:2018}. Central limit
theorems for the volume were subsequently obtained in
\cite{BesauThaele2020}, while asymptotic upper bounds for the variance of
the volume and the vertex number were recently derived in
\cite{FodorGruenfelder2025}. Related results in more general projective
Riemannian and Finsler geometries were developed in
\cite{BLW:2018,BesauRosenThaele2021}.
In these works, the underlying deterministic
sets are compact convex bodies. Thus the corresponding random polytopes are
bounded objects. Moreover, in the high-intensity regime the relevant
geometry is concentrated close to the boundary of the underlying body. After
a suitable local rescaling, this boundary geometry becomes asymptotically
flat, and one therefore observes phenomena which are closely related to
the classical Euclidean theory of random polytopes.

The present paper takes a new and different point of view. We consider random convex
hulls in an unbounded convex subset of hyperbolic space. More precisely, we work with horoballs in the $d$-dimensional hyperbolic space $\HH^d$, $d\geq 2$, of constant curvature
$-1$. In the upper halfspace model of $\HH^d$, in which $\HH^d$ is identified with the product space $\RR^{d-1}\times(0,\infty)$, the closed horoball based at the ideal
point $\infty$ is of the form
\[
	B_\lambda^\infty=\{(\bv,y)\in\HH^d:y\ge \lambda\},
	\qquad \lambda>0.
\]
It is unbounded, geodesically convex, and has exactly one distinguished ideal direction. From the perspective of hyperbolic geometry, a horoball is thus a natural container for a random convex hull which is unbounded but still has a controlled asymptotic direction at infinity. Since $B_\lambda^\infty$ has infinite hyperbolic volume, one cannot sample a finite number of uniformly distributed random points in it. To overcome this difficulty, we fix $\gamma>0$, let $\eta_\gamma$ be a stationary Poisson point process in $\HH^d$ of intensity $\gamma$, and restrict it to the horoball $B_\lambda^\infty$ by putting
\[
	\eta_{\gamma,\lambda}:=\eta_\gamma\cap B_\lambda^\infty.
\]
The main probabilistic object of this paper is the random hyperbolic polyhedron defined as the geodesic convex hull
\[
	K_{\gamma,\lambda}:=\conv(\eta_{\gamma,\lambda}) \subset B_\lambda^\infty
\]
of $\eta_{\gamma,\lambda}$. The geometric motivation for this model is twofold. First, horoballs are the canonical convex neighbourhoods of ideal boundary points and, after taking
quotients by parabolic subgroups, they provide standard local models for the non-compact ends of finite-volume hyperbolic manifolds. Second, horoballs arise as boundary limits of large geodesic balls. More precisely, if the centre of a geodesic ball tends to an ideal boundary point while its radius diverges in such a way that the signed distance of its boundary from a fixed reference point converges to a finite limit, then the balls
converge locally to a horoball. The present construction may therefore be viewed as a limiting model for random polytopes in hyperbolic balls of growing radius, observed in a moving frame near a boundary point escaping to infinity. In this sense, $K_{\gamma,\lambda}$ is a random-polytope model ``at infinity'', see Figure~\ref{fig:simulationpoincare} for simulations in the Poincar\'e ball model.

\begin{figure}
	\centering
	\begin{tikzpicture}
		\begin{scope}
			\clip (-3.7,-3.8) rectangle (3.7, 3.8);
			\node at (0,0) {
				\includegraphics[width=0.45\linewidth]{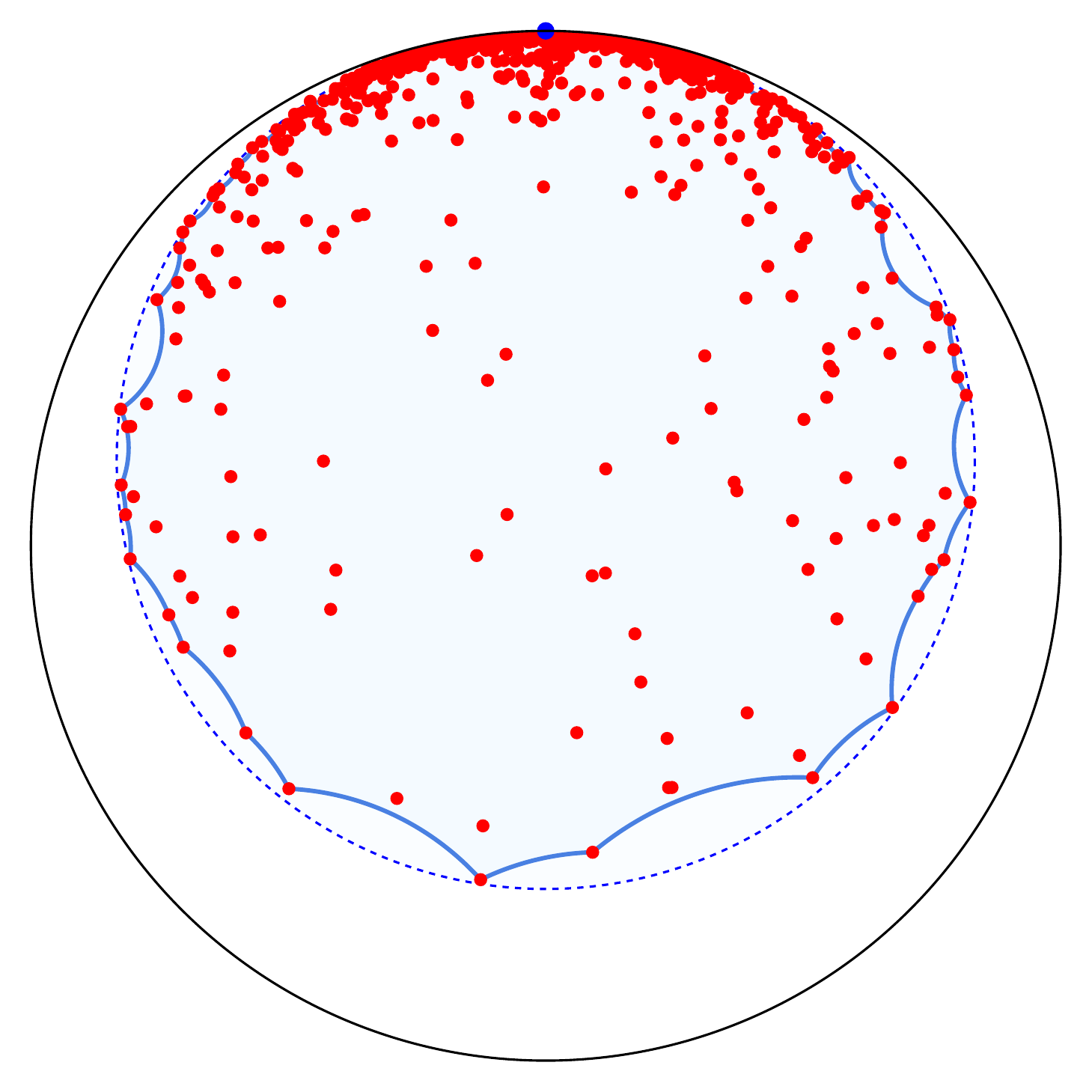}};
		\end{scope}
		\begin{scope}[xshift=8cm]
			\clip (-3.7,-3.8) rectangle (3.7,3.8);
			\node at (0, 0) {
				\includegraphics[width=0.78\linewidth]{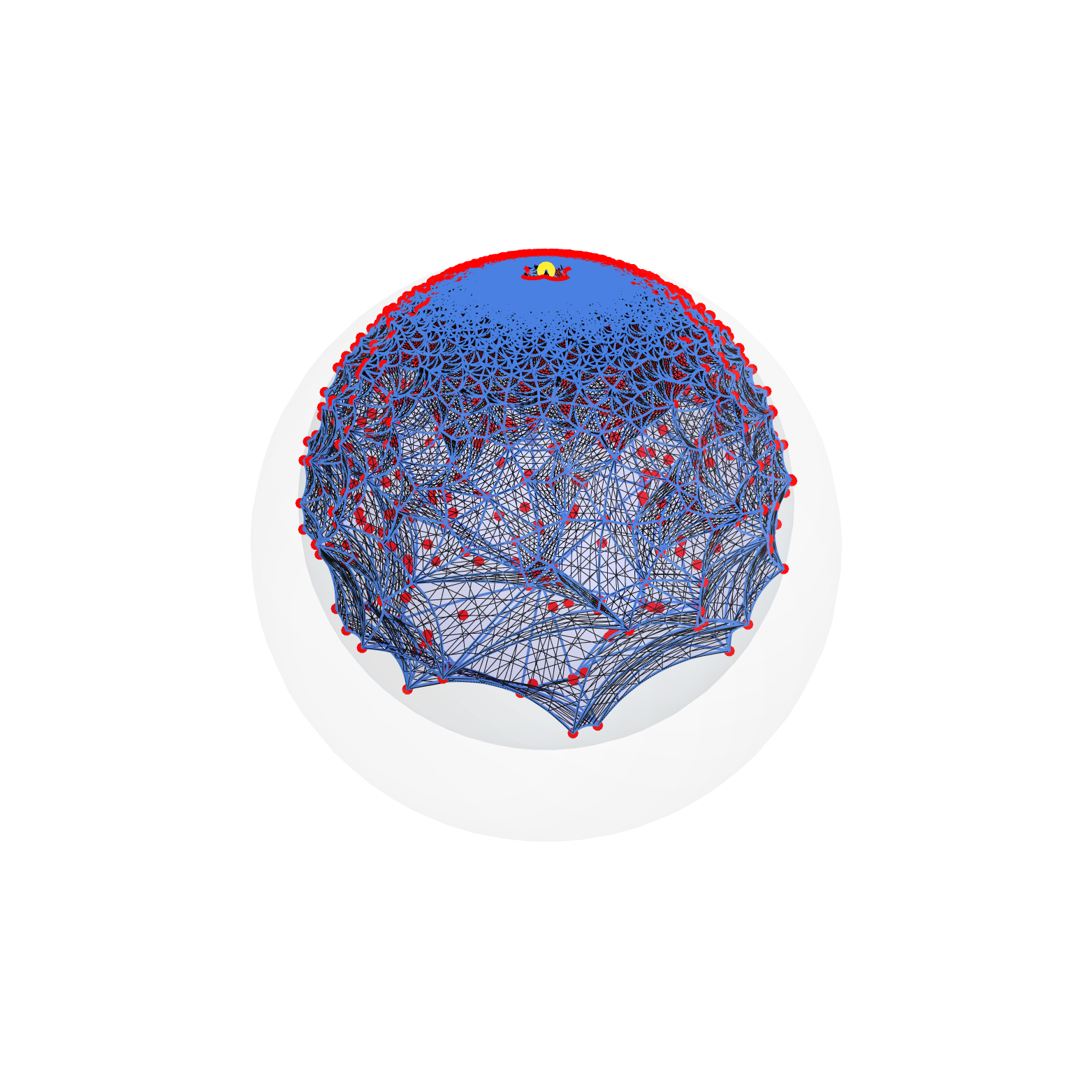}};
		\end{scope}
	\end{tikzpicture}
	\caption{Simulations of the random hyperbolic polyhedra $K_{\gamma,\lambda}$ for $d=2$ (left) and $d=3$ (right) in the Poincar\'e ball model of hyperbolic space. The parameters are $\lambda=0.2$ and $\gamma=5$ (left) or $\gamma=1.5$ (right).}
	\label{fig:simulationpoincare}
\end{figure}

To analyse the local geometry of $K_{\gamma,\lambda}$, we project
its facets to the horosphere which bounds $B_\lambda^\infty$, and which in view of its intrinsic Euclidean geometry can be identified with
$\RR^{d-1}$. This gives rise to a stationary random tessellation of
$\RR^{d-1}$, which we denote by $\cD_{\gamma,\lambda}$. One of the basic
observations of the paper is that $\cD_{\gamma,\lambda}$ can be identified explicitly with a dual Poisson--Laguerre tessellation as investigated in \cite{GiWL}. The principal asymptotic question is
then the behaviour of this tessellation in the low-intensity regime,
in particular for the scaling
\[
	\gamma=\lambda^{d-1},
	\qquad \lambda\downarrow0.
\]
The precise construction of $\cD_{\gamma,\lambda}$, as well as the relevant notions from the upper
halfspace model, ideal boundary, horospheres and horoballs, will be
introduced formally in Section~\ref{sec:Prelim}.

Our model is also motivated by recent developments on ideal
Poisson--Voronoi tessellations (IPVTs) in hyperbolic spaces. It was shown in
\cite{DCE+26} that the low-intensity limit of Poisson--Voronoi
tessellations in $\HH^d$ gives rise to a non-degenerate limiting
tessellation whose cells are unbounded and possess a unique ideal endpoint.
Since then, several further aspects and extensions of the model have been
investigated, including IPVTs beyond real hyperbolic spaces
\cite{Dachille2024,DachilleGrebikKhezeliReckeWilkens2025,DachilleKhezeli2026}, indistinguishability properties of
their cells \cite{Mellick2024}, and explicit face-volume densities and
typical face volumes in the hyperbolic setting
\cite{DachilleThaele2026}. IPVTs have also
found striking applications, for example in connection with Gaboriau's fixed
price conjecture and with upper bounds for the Cheeger constant of
high-genus closed hyperbolic surfaces, see
\cite{Application1,Application2}. The present construction may be viewed as the canonical random-polytope
analogue of this ideal Poisson--Voronoi limit.  In both settings, the geometry of the random polyhedral set (i.e. zero cell of ideal Poisson--Voronoi tessellation)
is organised by a distinguished ideal direction and remains non-trivial in
a low-intensity regime. There is also a closer structural analogy. After
identifying the ideal endpoint of the IPVT zero cell with $\infty$ in the upper halfspace model, its boundary
can be described through a deposition model whose projection gives rise to
a Laguerre tessellation on $\RR^{d-1}$. In the present model, the boundary
facets of $K_{\gamma,\lambda}$ are likewise encoded by a stationary
tessellation obtained through horospherical projection. Here, however, the
basic random object is not a Voronoi cell in $\HH^d$, but the geodesic
convex hull of a Poisson process in a horoball.

\subsection{Overview of the main results}

We now describe informally the main results of the paper. Our first step is
an exact identification of the projected boundary structure of
$K_{\gamma,\lambda}$. The facets of the hyperbolic polyhedron
$K_{\gamma,\lambda}$ are supported by totally geodesic hypersurfaces, which
in the upper halfspace model are represented by Euclidean hemispheres
whose boundary is orthogonal to $\RR^{d-1}\times\{0\}$. Projecting such facets to the
horosphere, or equivalently to  $\RR^{d-1}$,
gives rise to a random simplicial tessellation of $\RR^{d-1}$, denoted by
$\cD_{\gamma,\lambda}$.

Our first main observation is that this tessellation is not an ad hoc
object, but belongs to the broad class of dual Poisson--Laguerre
tessellations as studied in \cite{GiWL}. In that framework the tessellation
is generated by a Poisson point process on $\RR^{d-1}\times\RR$ whose
intensity measure has a density of the form
\[
	(\bv,h)\mapsto f(h)
\]
with respect to product Lebesgue measure, where $f:\RR\to\RR$ is a non-negative
locally integrable height function satisfying the admissibility assumptions
of \cite[Definition~3.4]{GiWL}. In our model this height function can be
identified explicitly, see Theorem~\ref{thm:DlambdaIsTessellation}. Namely, after the transformation $(\bv,y)\mapsto(\bv,y^2)$,
the point process $\eta_{\gamma,\lambda}$ becomes a Poisson point process
on $\RR^{d-1}\times(0,\infty)$ with height density
\[
	f_{\gamma,\lambda}(h)
	=
	\frac{\gamma}{2}h^{-(d+1)/2}{\bf 1}\{h\ge\lambda^2\}.
\]

The second step is quantitative. We compute explicitly the cell intensity of
the random tessellation $\cD_{\gamma,\lambda}$. If
$\xi_{d-1}(\gamma,\lambda)$ denotes the intensity of full-dimensional
cells, then
Theorem~\ref{thm:CellIntensity} gives an explicit integral representation
for $\xi_{d-1}(\gamma,\lambda)$. The formula involves two components: a hypergeometric term, coming from a
Poisson void probability, and a geometric integral $I_d(s)$, which
encodes the mean volume of simplices with vertices in a Euclidean ball.
This computation is not a consequence of the Poisson--Laguerre
representation alone and requires the use of a paraboloid
Blaschke--Petkantschin transformation. Although we obtain an exact formula
only for the full-dimensional cell intensity, the subsequent local
convergence theorem, together with uniform stabilization estimates, allows
us to prove convergence of the intensities of all lower-dimensional faces
as well, see Corollary~\ref{cor:AllFaceIntensities}.

We then pass to the low-intensity regime, where $\gamma=\lambda^{d-1}$ and $\lambda\downarrow0$.
This scaling should be viewed as a horospherical, or cusp-type, scaling near
an ideal boundary point. Indeed, horospheres carry an intrinsic Euclidean
geometry, and this is also visible in the Laguerre representation. Under the
critical scaling $\gamma=\lambda^{d-1}$, the height density becomes
\[
	f_{\lambda^{d-1},\lambda}(h)
	=
	\frac{1}{2}\lambda^{d-1}h^{-(d+1)/2}{\bf 1}\{h\ge\lambda^2\},
\]
so that for every fixed $x>0$,
\[
	\int_{\lambda^2}^x
	f_{\lambda^{d-1},\lambda}(h)\,\dint h
	=
	\frac{1}{d-1}
	\left(
	1-\lambda^{d-1}x^{-\frac{d-1}{2}}
	\right)
	\longrightarrow
	\frac{1}{d-1},
	\qquad \lambda\downarrow0.
\]
Thus the height measure collapses to $\frac{1}{d-1}$ times the Dirac measure at zero, meaning that the height coordinates of the input
point process collapse to zero, and one obtains the classical Poisson--Delaunay tessellation in
$\RR^{d-1}$ as a scaling limit. More precisely, using general convergence criteria for Laguerre tessellations obtained in \cite{GiWL26} we deduce that, under a suitable coupling, the skeleton
$\skel(\cD_{\lambda^{d-1},\lambda})$ of the dual Poisson--Laguerre
tessellation $\cD_{\lambda^{d-1},\lambda}$ converges locally to the
skeleton of the Poisson--Delaunay tessellation
$\cD_{1/(d-1)}$. That is, for every ball $\BB_R^{d-1}$ in $\RR^{d-1}$ of radius $R>0$, it holds that
\[
	\PP\left(
	\skel(\cD_{\lambda^{d-1},\lambda})\cap \BB_R^{d-1}
	=
	\skel(\cD_{1/(d-1)})\cap \BB_R^{d-1}
	\right)
	\longrightarrow 1,
	\qquad \lambda\downarrow0,
\]
see Theorem~\ref{thm:SkeletonConvergence}. This local coupling provides
information of a different nature than convergence of the cell intensity
alone, since it controls the tessellation itself inside bounded observation
windows. It implies convergence in probability of all bounded-window statistics that are determined by the tessellation inside a fixed
observation window. Combined with uniform moment bounds for local face
counts, it also yields convergence of all face intensities, as already anticipated above. More precisely,
if $\xi_k(\gamma,\lambda)$ denotes the intensity of the
$k$-dimensional faces of $\cD_{\gamma,\lambda}$, then
\[
	\xi_k(\lambda^{d-1},\lambda)
	\longrightarrow
	\xi_k^{\operatorname{PD}}\Big(\frac{1}{d-1}\Big),
	\qquad k\in\{0,\ldots,d-1\},
\]
where $\xi_k^{\operatorname{PD}}(\frac{1}{d-1})$ denotes the
$k$-face intensity of the limiting Poisson--Delaunay tessellation in
$\RR^{d-1}$. Moreover, together with convergence of the
full-dimensional cell intensity, the local convergence theorem implies weak convergence of the typical cell of
$\cD_{\lambda^{d-1},\lambda}$ to the typical cell of the
Poisson--Delaunay tessellation as follows from \cite{GKT24}.

A further result concerns the volume, one of the classical functionals in
the theory of random polytopes. In our horoball model the global volume of
$K_{\gamma,\lambda}$, as well as the global missed volume inside
$B_\lambda^\infty$, is infinite. We therefore localize in the horospherical direction. For a bounded and open
subset $W\subset\RR^{d-1}$, which we identify with a set of ideal points of $\HH^d$ different from $\infty$, we consider the unbounded horospherical cylinder $C_\lambda(W)$ that is the intersection of $B_\lambda^\infty$ with the union of all geodesics joining a point of $W$ to $\infty$.
Note that these geodesics meet the boundary of $B_\lambda^\infty$ orthogonally, and that for $W_\lambda:=\partial B_\lambda^\infty\cap C_\lambda(W)$ we have
\begin{equation}
	\label{eqn:VolumeCylinder}
	\cH_{\HH^d}^{d-1}(W_\lambda) = (d-1)\, \vol_{\HH^d}\bigl(C_\lambda(W)\bigr)  = \vol_{\RR^{d-1}}\bigl(\tfrac{1}{\lambda} W\bigr).
\end{equation}
We consider the normalised
local volume and surface-area functionals
\begin{equation}\label{def:local_vol}
	V_{\gamma,\lambda}(W)
	:=
	\frac{
		\vol_{\HH^d}\bigl(
		K_{\gamma,\lambda}\cap C_\lambda(W)
		\bigr)
	}{
		\vol_{\HH^d}\bigl(C_\lambda(W)\bigr)
	},
	\qquad
	S_{\gamma,\lambda}(W)
	:=
	\frac{
		\cH_{\HH^d}^{d-1}\bigl(
		\partial K_{\gamma,\lambda}\cap C_\lambda(W)
		\bigr)
	}{
		\cH_{\HH^d}^{d-1}\bigl(
		\partial B_\lambda^\infty \cap C_\lambda(W)
		\bigr)
	}.
\end{equation}
In Theorem~\ref{thm:LocalVolume}, we derive exact integral representations
for both
$\EE V_\gamma = \EE V_{\gamma,\lambda}(W)$ and
$\EE S_\gamma = \EE S_{\gamma,\lambda}(W)$, and also make explicit that these expected values are independent of $W$ and $\lambda$.
These formulas are parallel to the
cell-intensity formula in Theorem~\ref{thm:CellIntensity}, but contain
additional geometric factors associated with the supporting simplex of a
projected cell. In the volume case, this factor is the hyperbolic volume of the simplex with one ideal vertex associated
with the projected cell, whereas in the
surface-area case it is the intrinsic hyperbolic $(d-1)$-volume of the
facet contained in the supporting totally geodesic hypersurface.

We further establish in Corollary~\ref{cor:Efron} an Efron-type identity which relates the vertex
intensity of the projected tessellation to the expected local missed
volume. More precisely, we show that
\[
	\lambda^{d-1}\,\xi_0(\gamma,\lambda)
	=
	\frac{\gamma}{(d-1)}\,
	\bigl( 1 - \EE V_{\gamma}\bigr).
\]
This identity is the hyperbolic analogue of the classical relation between
the expected number of vertices of a random polytope and its expected missed
volume. The behaviour of the volume and surface-area functionals in the
present model is, however, qualitatively different from that in the usual
Euclidean theory. For random polytopes in a fixed Euclidean convex body,
the high-intensity limit fills the underlying container, and the volume and
surface area converge to the corresponding functionals of that body. In the
present unbounded setting the hyperbolic volume of the
observation cylinder diverges as $\lambda\downarrow0$, whereas, in the
critical regime $\gamma=\lambda^{d-1}$, the expected unnormalised local
volume and, for $d\ge3$, also the expected unnormalised local surface area
converge to finite limits proportional to
$\lambda^{d-1}\,\mathcal{H}_{\HH^d}^{d-1}(W_\lambda)=\vol_{\RR^{d-1}}(W)$. The corresponding constants admit geometric
representations in terms of volumes of ideal hyperbolic simplices and their
facets. The case of dimension $d=2$ is exceptional. In this case, the limiting facets are ideal geodesic
segments of infinite hyperbolic length, and the expected local boundary
length grows logarithmically as $\lambda\downarrow0$. After the
respective normalizations, both the volume and surface-area ratios
therefore tend to zero. Finally, the Efron-type identity converts the
local-volume asymptotics into a second-order expansion for the vertex
intensity.

\section{Preliminaries}\label{sec:Prelim}

\subsection{Notation}

We write $\|\cdot\|$ for the Euclidean norm and $\vol_{\RR^{d-1}}$ for the Lebesgue measure in $\RR^{d-1}$. If
$\bx_1,\ldots,\bx_m\in\RR^{d-1}$, then $[\bx_1,\ldots,\bx_m]$
denotes their Euclidean convex hull. The $(d-1)$-dimensional Euclidean
volume of the simplex $[\bx_1,\ldots,\bx_d]$ is denoted by
\[
	\Delta_{d-1}(\bx_1,\ldots,\bx_d)
	:=
	\vol_{\RR^{d-1}}([\bx_1,\ldots,\bx_d]).
\]
We denote for $m\in\NN$ by $\kappa_m=\frac{\pi^{m/2}}{\Gamma(1+\frac{m}{2})}$ the volume of the $m$-dimensional Euclidean unit ball $\BB^m$, and by $\sigma$ the spherical Lebesgue measure on
$\SS^{m-1}:=\partial\, \BB^m$. It is normalised in such a way that $\sigma(\SS^{m-1})=m\kappa_m$. Furthermore, $o$ denotes the origin of $\RR^{d-1}$, and
$\BB_R^{d-1}$ denotes the Euclidean ball in $\RR^{d-1}$ with centre $o$ and
radius $R>0$. For a point process $\zeta$ and $m\in\NN$, we write
$(\zeta)_{\neq}^m$ for the set of $m$-tuples of pairwise distinct points
of $\zeta$.

For a subset $A$ of a topological space we write $\partial A$, $\operatorname{cl}\,A$ and $\operatorname{int}\,A$ for the topological boundary, closure and interior of $A$. Moreover, the restriction of a measure $\mu$ to a measurable set $A$ is denoted by $\mu|_A$. The cardinality of a set $A$ is denoted by $\#A$.

We shall use the Gauss hypergeometric function. For $a\in\CC$ and
$n\in\NN$, the rising factorial is denoted by $a^{\overline{n}}:=a(a+1)\cdots(a+n-1)$, $n\geq 1$,
and we put $a^{\overline{0}}:=1$. If $a\notin\{0,-1,-2,\ldots\}$, then $a^{\overline{n}}=\frac{\Gamma(a+n)}{\Gamma(a)}$.
For $a,b,c\in\CC$ with $c\notin\{0,-1,-2,\ldots\}$, the Gauss
hypergeometric function is defined on the unit disc by the series
\[
	{}_2F_1(a,b;c;z)
	=
	\sum_{n=0}^{\infty}
	\frac{a^{\overline{n}}b^{\overline{n}}}{n!c^{\overline{n}}}z^n,
	\qquad |z|<1,
\]
and elsewhere by analytic continuation whenever this is well defined.

\subsection{Hyperbolic space and upper halfspace model}

In this paper we consider the $d$-dimensional hyperbolic space $\HH^d$ of constant sectional curvature $-1$. We work with the upper halfspace model, in which $\HH^d$ is identified with the product space $\RR^{d-1}\times(0,\infty)$. Points are represented as tuples $(\bv,y)$, where $\bv=(v_1,\ldots,v_{d-1})\in\RR^{d-1}$ and $y\in(0,\infty)$. We refer to $\bv$ as the spatial coordinate and to $y$ as the height coordinate of the point $(\bv,y)\in\HH^d$.   In this representation, the hyperbolic metric is given by
\begin{equation}\label{eq:HyperbolicMetric}
	\dint s^2 = y^{-2}\left(\dint v_1^2+\ldots+\dint v_{d-1}^2+\dint y^2\right).
\end{equation}
The associated Riemannian volume measure on $\HH^d$ is therefore
\[
	\dint\!\vol_{\HH^d}(\bv,y)
	= y^{-d}\,\dint\bv\dint y = y^{-d} \,\dint v_1\cdots\dint v_{d-1}\dint y.
\]
In particular, for any measurable set $A\subset\HH^d$, its hyperbolic volume is given by
\begin{equation}\label{eq:HyperbolicVolume}
	\vol_{\HH^d}(A)
	= \int_A y^{-d}\,\dint v_1\cdots\dint v_{d-1}\,\dint y.
\end{equation}
Moreover, we denote by $\cH^{d-1}_{\HH^d}$ the $(d-1)$-dimensional hyperbolic Hausdorff measure.
In what follows we will also use the Lebesgue measure on $\RR^{d-1}$, namely
\[
	\dint\!\vol_{\RR^{d-1}}(\bv) = \dint\bv
	= \dint v_1\cdots\dint v_{d-1}.
\]
For background material on hyperbolic geometry and especially the upper halfspace model we refer the reader to \cite{BenedettiPetronio,Ratcliffe} and \cite[Chapter 36]{Voight}.

\subsection{Geodesics, horospheres and convexity}

Geodesic lines in $\HH^d$ admit a simple geometric description in the upper halfspace model. In this model the ideal boundary is represented by
the extended Euclidean space
\[
	\partial_\infty \HH^d
	:=
	\RR^{d-1}\cup\{\infty\}.
\]
The points of $\RR^{d-1}$ correspond to the ideal endpoints approached as
$y\downarrow0$, while the additional point $\{\infty\}$ corresponds to the
ideal endpoint approached along vertical curves as $y\to\infty$. With this convention, the geodesic lines in $\HH^d$ are precisely
of the following two types. First, for every $\bw\in\RR^{d-1}$, the vertical
Euclidean ray $\{(\bw,y):y>0\}$
is a geodesic line. Its two ideal endpoints are $\bw$ and $\infty$. Second,
there are the Euclidean semicircles contained in
$\RR^{d-1}\times(0,\infty)$ which meet $\RR^{d-1}\times\{0\}$ orthogonally. Their two ideal endpoints are the two Euclidean endpoints of
the semicircle, viewed as points of
$\RR^{d-1}\subset\partial_\infty\HH^d$. In particular, every geodesic line in the upper halfspace model is determined by its two distinct ideal endpoints in $\partial_\infty\HH^d$.

Totally geodesic hypersurfaces in $\HH^d$ play the role of affine
hyperplanes in Euclidean geometry. In the upper halfspace model they are
precisely of the following two types. They are either vertical Euclidean
hyperplanes of the form
\[
	\{(\bv,y)\in\HH^d:\langle \bv,\bu\rangle=a\},
	\qquad \bu\in\SS^{d-2},\ a\in\RR,
\]
or Euclidean hemispheres contained in $\RR^{d-1}\times(0,\infty)$ which
meet $\RR^{d-1}\times\{0\}$
orthogonally. The ideal boundary of a vertical totally geodesic hypersurface
contains the point $\infty$, whereas the ideal boundary of a hemispherical
totally geodesic hypersurface is a Euclidean $(d-2)$-sphere in
$\RR^{d-1}\subset\partial_\infty\HH^d$. As in Euclidean space, every totally geodesic hypersurface separates
$\HH^d$ into two connected components. These components are called
hyperbolic halfspaces. Thus, in the upper halfspace model, hyperbolic
halfspaces are bounded either by vertical Euclidean hyperplanes or by
Euclidean hemispheres orthogonal to $\RR^{d-1}\times\{0\}$.

Geodesic spheres in $\HH^d$ are sets of points at fixed hyperbolic distance
from a centre point in $\HH^d$. In the upper halfspace model they are represented by Euclidean spheres contained in $\RR^{d-1}\times(0,\infty)$, although their Euclidean centres and radii are nonlinear functions of the hyperbolic centre and radius.

Horospheres are limiting analogues of geodesic spheres whose centres tend to
an ideal boundary point $\bu\in\partial_\infty\HH^d$. In the upper halfspace model their Euclidean
description depends on the ideal point $\bu$. If $\bu=\infty$, then the
corresponding horospheres are the horizontal Euclidean hyperplanes
\[
	H_\lambda:=\{(\bv,y)\in\HH^d:y=\lambda\},
	\qquad \lambda>0.
\]
The horoballs based at $\infty$ are the regions
\[
	\{(\bv,y)\in\HH^d:y\ge \lambda\}.
\]
If $\bu\in\RR^{d-1}$, then the corresponding horospheres are Euclidean
spheres contained in $\RR^{d-1}\times(0,\infty)$ and tangent to
$\RR^{d-1}\times\{0\}$ at $(\bu,0)$. The horoballs based at $\bu$ are the
Euclidean balls in the upper halfspace bounded by these tangent spheres. In particular, geodesic spheres are compact hypersurfaces centred at points of
$\HH^d$, whereas horospheres are non-compact hypersurfaces associated with
ideal boundary points.

Next we recall two notions of convexity in hyperbolic space. A set
$C\subset\HH^d$ is called geodesically convex if, for any two points
$\bx,\by\in C$, the geodesic segment joining $\bx$ and $\by$ is
contained in $C$. Given a locally finite set $X\subset\HH^d$, its
geodesic convex hull is defined by
\[
	\conv(X)
	:=
	\bigcap_{\substack{C \text{ is geodesically convex }\\X\subseteq C}}C.
\]
Equivalently, $\conv(X)$ is the union of all geodesic simplices with
vertices in $X$. This is the direct analogue of the Euclidean convex hull,
with Euclidean line segments replaced by hyperbolic geodesic segments.

There is a second, stronger notion of convexity which is adapted to the
horospherical geometry of $\HH^d$. A set $C\subset\HH^d$ is called
horospherically convex, or $h$-convex, if it can be represented as an
intersection of horoballs. Accordingly, the $h$-convex hull of a set $X\subset\HH^d$ is defined as the smallest $h$-convex set containing $X$, namely
\[
	\conv_h(X)
	:=
	\bigcap_{\substack{B \text{ is a horoball }\\X\subseteq B}}B.
\]
Since every horoball is geodesically convex, every $h$-convex set is
geodesically convex, and consequently $\conv(X)\subseteq \conv_h(X)$.
The converse implication is false in general: a geodesically convex set need
not be an intersection of horoballs.

Note that, by definition, $\conv_h(X)$ is always closed, whereas $\conv(X)$ need not be closed if $X\subseteq\HH^d$ is unbounded, even if it is locally finite.

\section{Model and results}

\subsection{Random convex hulls in a horoball}

We now describe the random geometric model studied in this paper. Let $\gamma>0$ and let $\eta_{\gamma}$ be a Poisson point process on $\HH^d$ with intensity measure $\gamma\cdot\vol_{\HH^d}$. In particular, for every Borel set $A\subset\HH^d$ with $\vol_{\HH^d}(A)<\infty$, the random variable $\eta_{\gamma}(A)$ is Poisson distributed with mean $\gamma\cdot\vol_{\HH^d}(A)$, and the counts on disjoint sets are independent.

We fix $\lambda>0$ and consider the horosphere $H_\lambda$, which corresponds to the ideal boundary point at $\infty$ in the upper halfspace model. We restrict the Poisson point process to the horoball $B_\lambda^\infty$ bounded by $H_\lambda$. From here on, we work with the locally finite point process
\[
	\eta_{\gamma,\lambda}:=\eta_{\gamma}\cap B_\lambda^\infty.
\]
In the present setting, the horospherical convex hull of $\eta_{\gamma,\lambda}$ is trivial: since $B_\lambda^\infty$ is itself horospherically convex and contains $\eta_{\gamma,\lambda}$, we have
\[
	\conv_h(\eta_{\gamma,\lambda})=B_\lambda^\infty\qquad\text{a.s.}
\]
Indeed, since $\conv_h(\eta_{\gamma,\lambda})$ is horospherically convex and $\eta_{\gamma,\lambda}\subseteq B_\lambda^\infty$ we have $\conv_h(\eta_{\gamma,\lambda})\subseteq B_\lambda^\infty$. If $\conv_h(\eta_{\gamma,\lambda})\neq B_\lambda^\infty$, then there would be another horoball $B_{\lambda'}^\infty$ with $\lambda'>\lambda$, such that $\conv_h(\eta_{\gamma,\lambda})\subseteq B_{\lambda'}^\infty$. Since $\vol_{\HH^d}(B_{\lambda}^\infty\setminus B_{\lambda'}^\infty)=\infty$, with probability $1$ there is a point $(\bv,y)\in\eta_{\gamma}$ with $(\bv,y)\not\in B_{\lambda'}^\infty$, which leads to a contradiction. In contrast, the geodesic convex hull
\[
	K_{\gamma,\lambda}:=\conv(\eta_{\gamma,\lambda}),
\]
of $\eta_{\gamma,\lambda}$ is a highly nontrivial probabilistic object, see Figure~\ref{fig:d=2} for simulations in the upper halfspace model in the planar case $d=2$ and Figure \ref{fig:simulationpoincare} for simulations in the Poincar\'e ball model for $d=2$ and $d=3$.

\begin{figure}[t]
	\centering
	\includegraphics[width=0.45\columnwidth]{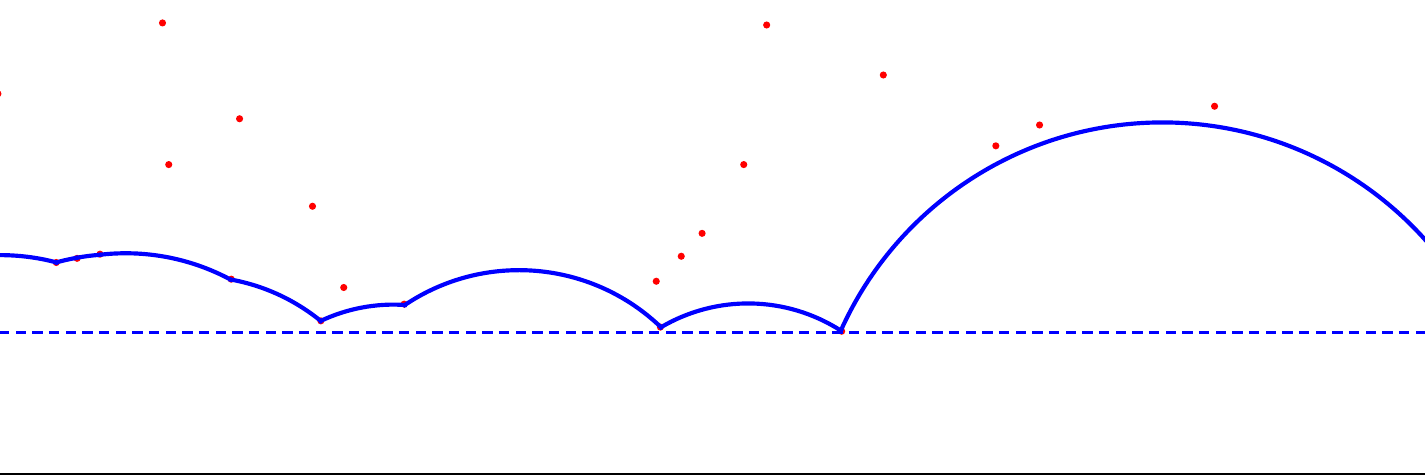}\qquad
	\includegraphics[width=0.45\columnwidth]{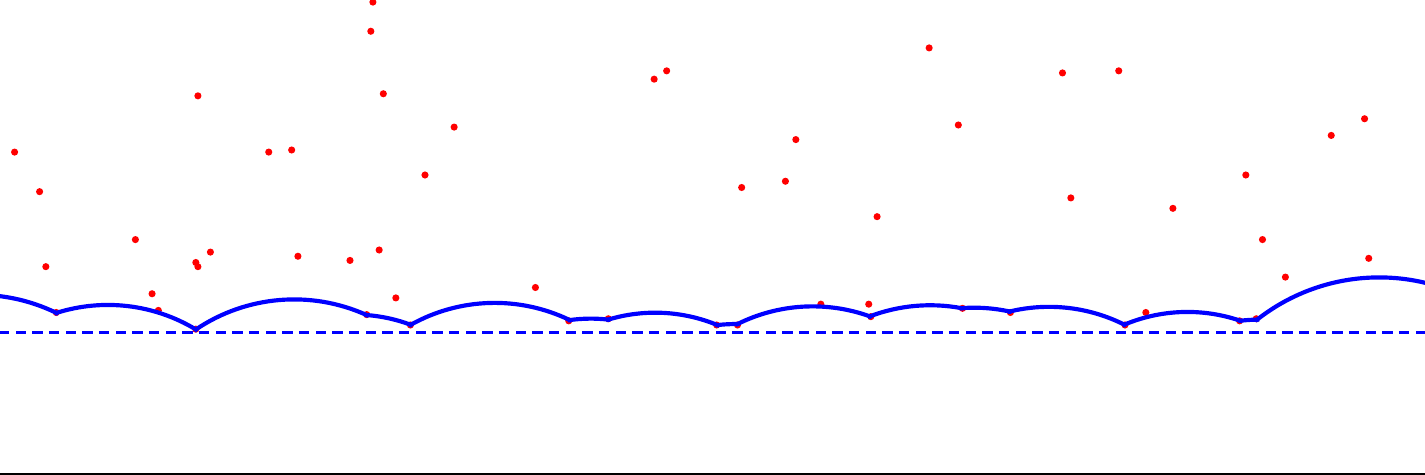}
	\caption{Simulation of a Poisson point process $\eta_{\gamma,\lambda}$ and the Poisson polyhedron $K_{\gamma,\lambda}$ for $d=2$ with $\gamma=1$ (left) and $\gamma=5$ (right).}
	\label{fig:d=2}
\end{figure}

The present setting is closely related to classical Euclidean models of
random polytopes, but also exhibits fundamental differences. In Euclidean
space, one often studies the convex hull of a homogeneous Poisson point
process restricted to a bounded convex body, leading to random polytopes
whose facial structure and asymptotic geometry depend on both the ambient
container set and the intensity of the process. If, on the other hand, the
Poisson process is restricted to a Euclidean halfspace, then the convex hull
is essentially trivial: almost surely, its closure coincides with the entire
halfspace. In hyperbolic space the situation is different. The horoball
$B_\lambda^\infty$ is itself a natural convex container, but it is bounded
by a horosphere and not by a totally geodesic hypersurface. Consequently,
there are two natural convexity structures to distinguish. As noted above,
the horospherical, or $h$-convex, hull of $\eta_{\gamma,\lambda}$ is
trivial, while the geodesic convex hull $K_{\gamma,\lambda}$ is
nontrivial. Its boundary is supported by totally geodesic hypersurfaces,
and the negative curvature of $\HH^d$ prevents this geodesic convex hull
from filling the horoball.

A further notable distinction from the classical Euclidean theory of random
polytopes is that $K_{\gamma,\lambda}$ is unbounded. Its unboundedness,
however, is highly structured. We shall use the term hyperbolic
polyhedron for a closed, geodesically convex subset of $\HH^d$ whose
boundary is supported by totally geodesic hypersurfaces and whose set of faces is
locally finite. Here local finiteness means that every compact subset of
$\HH^d$ meets only finitely many faces. Thus such a polyhedron may have
infinitely many faces globally, in contrast to a finite hyperbolic polytope.

For the random set $K_{\gamma,\lambda}$, the face structure is indeed
infinite. This reflects the fact that the restricted Poisson process
$\eta_{\gamma,\lambda}$ is locally finite but unbounded in the
horospherical directions. At the same time, all unbounded directions of
$K_{\gamma,\lambda}$ point towards the same ideal boundary point, namely
the point $\infty$ associated with the horoball $B_\lambda^\infty$. In this
sense the convex hull is one-ended.

\begin{theorem}\label{thm:UnboundedOneEnded}
	With probability one, $K_{\gamma,\lambda}$ is a closed unbounded hyperbolic
	polyhedron with locally finite but countably infinite face structure.
	Moreover, its closure in $\HH^d\cup\partial_\infty\HH^d$ satisfies $\operatorname{cl}(K_{\gamma,\lambda})\cap\partial_\infty\HH^d=\{\infty\}$.
\end{theorem}

As a direct consequence, purely combinatorial characteristics of $K_{\gamma,\lambda}$, such as the total number of $k$-faces, are almost surely infinite and therefore not informative. In particular, counting faces and taking expectations does not lead to meaningful quantities in this unbounded setting. Instead, one must adopt a local point of view. Such local descriptions are better suited to capture the geometry of $K_{\gamma,\lambda}$ and will form the basis of our analysis.

\begin{remark}\label{rmk:Lambda_iso}
	Let us remark on the role of $\lambda>0$ in our model. First, for $\lambda'>\lambda$ the Hausdorff distance between $B^\infty_{\lambda}$ and $B^\infty_{\lambda'}$ is $\log \frac{\lambda'}{\lambda}$ and we have $B^\infty_{\lambda'}\subset B^\infty_{\lambda}$. In particular, for a fixed realization of $\eta_{\gamma}$ we have
	$\eta_{\gamma,\lambda'}\subset \eta_{\gamma,\lambda}$ and therefore $K_{\gamma,\lambda'}\subset K_{\gamma,\lambda}$.
	Furthermore, $B^\infty_{\lambda}$ and $B^\infty_{\lambda'}$ are congruent and in the halfspace model the isometry that maps $B^\infty_\lambda$ to $B^\infty_{\lambda'}$ is the radial rescaling $\bx\mapsto \frac{\lambda'}{\lambda}\bx$, that is, $\frac{\lambda'}{\lambda} B^\infty_{\lambda} = B^\infty_{\lambda'}$.
	Since $\eta_{\gamma}$ is isometry invariant, we find that $\frac{1}{\lambda} K_{\gamma,\lambda}$ and $K_{\gamma,1}$ are equivalent in distribution.
\end{remark}

\subsection{Relation to dual Poisson--Laguerre tessellations}

Instead of studying the local geometry of $K_{\gamma,\lambda}$ directly in
the upper halfspace model, we pass to a Euclidean tessellation obtained by
projecting the relevant boundary facets to $\RR^{d-1}$. For this purpose, let $(\bv_1,y_1),\ldots,(\bv_d,y_d)\in(\eta_{\gamma,\lambda})_{\neq}^d$. We denote by $B(\bv_1,y_1,\ldots,\bv_d,y_d)$ the a.s.\ uniquely determined closed Euclidean halfball in
$\RR^{d-1}\times(0,\infty)$ which has a centre $(\bw,0)\in \RR^{d-1}\times \{0\}$ and whose boundary contains the points
$(\bv_1,y_1),\ldots,(\bv_d,y_d)$, see the left panel of Figure~\ref{fig:TransformationT}. We then define
$\cD_{\gamma,\lambda}$ as the collection of all projected simplices
\[
	[\bv_1,\ldots,\bv_d]:=\conv\{\bv_1,\ldots,\bv_d\}\subset\RR^{d-1}
\]
for which this halfball is empty of further points of the process. That is,
\begin{align*}
	\cD_{\gamma,\lambda}
	:=
	\Big\{
	 & [\bv_1,\ldots,\bv_d]\subset\RR^{d-1}:
	(\bv_1,y_1),\ldots,(\bv_d,y_d)\in
	(\eta_{\gamma,\lambda})_{\neq}^d,                          \\
	 & \qquad \bv_1,\ldots,\bv_d \text{ affinely independent},
	\;\operatorname{int}B(\bv_1,y_1,\ldots,\bv_d,y_d)
	\cap\eta_{\gamma,\lambda}
	=
	\varnothing
	\Big\}.
\end{align*}
This is the projected tessellation associated with the boundary facets of
$K_{\gamma,\lambda}$.

Next, we recall from \cite[Section~3.1.4]{GiWL} the construction of the
dual Laguerre tessellation in the form needed below. Let $\zeta$ be a
locally finite point set in $\RR^{d-1}\times\RR$. For distinct points
\[
	\bz_i=(\bv_i,h_i)\in \RR^{d-1}\times\RR,
	\qquad i\in\{1,\ldots,d\},
\]
whose spatial coordinates $\bv_1,\ldots,\bv_d$ are affinely independent, there
is a unique translate of the standard downward paraboloid
\[
	\Pi:=\{(\bv,h)\in\RR^{d-1}\times\RR:h\le -\|\bv\|^2\}
\]
whose boundary contains $\bz_1,\ldots,\bz_d$. We denote this paraboloid by
$\Pi(\bz_1,\ldots,\bz_d)$, that is,
\[
	\Pi(\bz_1,\ldots,\bz_d)
	=
	\{(\bv,h)\in\RR^{d-1}\times\RR:
	h\le s-\|\bv-\bw\|^2\}
\]
for uniquely determined $\bw=\bw(\bz_1,\dotsc,\bz_d)\in\RR^{d-1}$ and $s=s(\bz_1,\dotsc,\bz_d)\in\RR$. The point
$(\bw,s)$ is called the apex of the paraboloid, and we write $A(\bz_1,\ldots,\bz_d):=\bw$
for its spatial coordinate, see the right panel of Figure~\ref{fig:TransformationT}.
The dual Laguerre tessellation generated by $\zeta$ is the collection of
simplices
\begin{align*}
	\mathcal L^*(\zeta)
	:=
	\Big\{
	 & [\bv_1,\ldots,\bv_d]:
	\bz_i=(\bv_i,h_i)\in\zeta,\ i\in\{1,\ldots,d\},\\
	 & \qquad \bv_1,\ldots,\bv_d \text{ affinely independent},
	\ \operatorname{int}\Pi(\bz_1,\ldots,\bz_d)\cap\zeta=\varnothing
	\Big\}.
\end{align*}
For the Poisson point processes
considered below this collection is a locally finite simplicial random
tessellation of $\RR^{d-1}$, see \cite[Section~3.1.4]{GiWL}. In this case
we call $\mathcal L^*(\zeta)$ the dual Poisson--Laguerre tessellation
generated by $\zeta$.

\begin{figure}[t]
	\centering
	\begin{tikzpicture}[scale=1.05,>=stealth]

		\def\c{2.4}      
		\def\R{1.8}      
		\def\xone{1.15}  
		\def\xtwo{3.35}  
		\def\sep{8.2}    

		\pgfmathsetmacro{\yone}{sqrt(\R^2-(\xone-\c)^2)}
		\pgfmathsetmacro{\ytwo}{sqrt(\R^2-(\xtwo-\c)^2)}
		\pgfmathsetmacro{\hone}{\yone*\yone}
		\pgfmathsetmacro{\htwo}{\ytwo*\ytwo}
		\pgfmathsetmacro{\Rsq}{\R*\R}
		\pgfmathsetmacro{\xmin}{\c-\R}
		\pgfmathsetmacro{\xmax}{\c+\R}

		\begin{scope}

			\draw[->] (-0.6,0) -- (5.3,0);
			\draw[->] (-0.3,-0.2) -- (-0.3,4.0);


			\draw[very thick,blue,domain=\xmin:\xmax,samples=100]
			plot (\x,{sqrt(\R*\R-(\x-\c)^2)});

			\filldraw[red] (\xone,\yone) circle (1.5pt) node[above left] {$(v_1,y_1)$};
			\filldraw[red] (\xtwo,\ytwo) circle (1.5pt) node[above right] {$(v_2,y_2)$};

			\draw[densely dashed] (\xone,0) -- (\xone,\yone);
			\draw[densely dashed] (\xtwo,0) -- (\xtwo,\ytwo);
			\fill (\xone,0) circle (1.1pt) node[below] {$v_1$};
			\fill (\xtwo,0) circle (1.1pt) node[below] {$v_2$};

			\fill (\c,0) circle (1.1pt) node[below] {$w$};

			\draw[densely dashed] (\c,0) -- (\c,\R);


		\end{scope}

		\draw[->,thick] (6,1.35) -- (7,1.35)
		node[midway,above] {$T$};

		\begin{scope}[shift={(\sep,0)}]

			\draw[->] (-0.6,0) -- (5.3,0);
			\draw[->] (-0.3,-0.2) -- (-0.3,4.0);


			\draw[very thick,blue,domain=\xmin:\xmax,samples=100]
			plot (\x,{\Rsq-(\x-\c)^2});

			\filldraw[red] (\xone,\hone) circle (1.5pt) node[above left] {$\bz_1\,=\,(v_1,h_1)$};
			\filldraw[red] (\xtwo,\htwo) circle (1.5pt) node[above right] {$\bz_2\,=\,(v_2,h_2)$};

			\draw[densely dashed] (\xone,0) -- (\xone,\hone);
			\draw[densely dashed] (\xtwo,0) -- (\xtwo,\htwo);
			\fill (\xone,0) circle (1.1pt) node[below] {$v_1$};
			\fill (\xtwo,0) circle (1.1pt) node[below] {$v_2$};

			\fill (\c,0) circle (1.1pt) node[below] {$w$};
			\filldraw[blue] (\c,\Rsq) circle (1.5pt); 
			\draw[densely dashed] (\c,0) -- (\c,\Rsq);


		\end{scope}
	\end{tikzpicture}
	\caption{Left: Two points $(v_1,y_1),(v_2,y_2)\in\HH^2$ with the ball $B(v_1,y_1,v_2,y_2)$ and its centre $w$. Right: The same picture after the transformation $T$ with $T(v_1,y_1)=\bz_1$ and $T(v_2,y_2)=\bz_2$. The half-circle becomes the parabola $\Pi(\bz_1,\bz_2)$ with spatial apex coordinate $w=A(\bz_1,\bz_2)$.}
	\label{fig:TransformationT}
\end{figure}
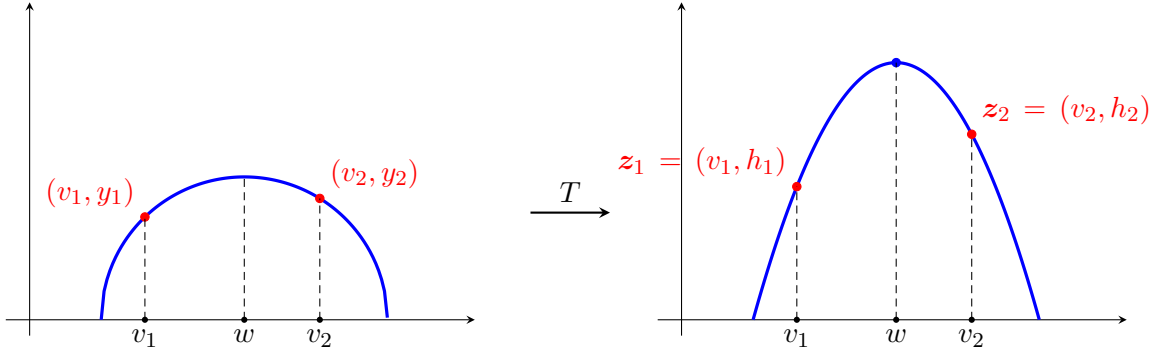

We now connect this construction with the random set
$K_{\gamma,\lambda}$. Consider the transformation
\[
	T:\RR^{d-1}\times(0,\infty)\to\RR^{d-1}\times(0,\infty),
	\qquad
	T(\bv,y):=(\bv,y^2).
\]
By the mapping theorem for Poisson point processes
\cite[Theorem~5.1]{LPbook}, the image process $\widetilde\eta_{\gamma,\lambda}:=T(\eta_{\gamma,\lambda})$
is a Poisson point process on $\RR^{d-1}\times(0,\infty)$ with intensity
measure $\Lambda_{\gamma,\lambda}$ given by
\begin{equation}\label{eq:IntensityEta}
	\dint\Lambda_{\gamma,\lambda}(\bv,h)
	=
	\frac{\gamma}{2}h^{-(d+1)/2}
	\mathbf 1\{h\ge \lambda^2\}\,
	\dint\bv\dint h .
\end{equation}
The density in \eqref{eq:IntensityEta} is admissible in the sense of
\cite[Definition~3.4]{GiWL}, and hence the theory of
Poisson--Laguerre tessellations applies to $\widetilde\eta_{\gamma,\lambda}$.

The following result identifies the projected tessellation obtained from
hyperbolic empty halfballs with the dual Poisson--Laguerre tessellation
generated by the transformed process $\widetilde\eta_{\gamma,\lambda}$.

\begin{theorem}\label{thm:DlambdaIsTessellation}
	For every $\gamma,\lambda>0$, the random collection
	$\mathcal D_{\gamma,\lambda}$ coincides almost surely with the dual
	Poisson--Laguerre tessellation generated by
	$\widetilde\eta_{\gamma,\lambda}$. More precisely,
	\[
		\mathcal D_{\gamma,\lambda}
		=
		\mathcal L^*(\widetilde\eta_{\gamma,\lambda})
		\qquad \text{a.s.},
	\]
	where $\widetilde\eta_{\gamma,\lambda}$ has intensity measure
	$\Lambda_{\gamma,\lambda}$ given by \eqref{eq:IntensityEta}.
\end{theorem}

Figure~\ref{fig:d=3} shows a realization of $K_{\gamma,\lambda}$ for $d=3$, the view
from above in the right panel displays the induced tessellation
$\cD_{\gamma,\lambda}$ of $\RR^{2}$.

There is a close formal analogy between the present construction and the
deposition model used in \cite[Sections~5.1 and~5.3.2]{DCE+26} to describe
the zero cell of the ideal Poisson--Voronoi tessellation. After sending the
unique ideal endpoint of that cell to $\infty$, its boundary is generated
by a Poisson family of Euclidean hemispheres in the upper halfspace model.
Passing to the associated stationary deposition model and projecting its
boundary orthogonally onto $\RR^{d-1}$ yields a stationary Laguerre
tessellation. This tessellation belongs to the $\beta'$-class with parameter $\beta=d$ introduced in \cite{GKT22}.

A similar empty-hemisphere mechanism occurs in the present setting. Under
the transformation $(\bv,y)\mapsto(\bv,y^2)$ the Euclidean hemispheres supporting $K_{\gamma,\lambda}$
become the surface of downward paraboloids, and the transformed Poisson process has intensity density $(\bv,h)\mapsto\frac{\gamma}{2}\,h^{-\frac{d+1}{2}}{\bf 1}\{h\ge\lambda^2\}$ by \eqref{eq:IntensityEta}. If the truncation at $\lambda^2$ were formally omitted, this power law
would correspond to a $\beta$-model of random Laguerre tessellation introduced in \cite{GKT22}, but with parameter $\beta=-\frac{d+1}{2}<-1$,
which lies outside the admissible range $\beta>-1$. The lower height cutoff is therefore
essential: it removes the non-integrable singularity at $h=0$ and makes
the driving Poisson process locally finite. Consequently,
$\cD_{\gamma,\lambda}$ is a well-defined dual Poisson--Laguerre
tessellation, but it does not belong to the $\beta$-Delaunay or $\beta'$-Delaunay family from \cite{GKT22}.

\subsection{The cell intensity}

\begin{figure}[t]
	\centering
	\includegraphics[width=0.55\columnwidth]{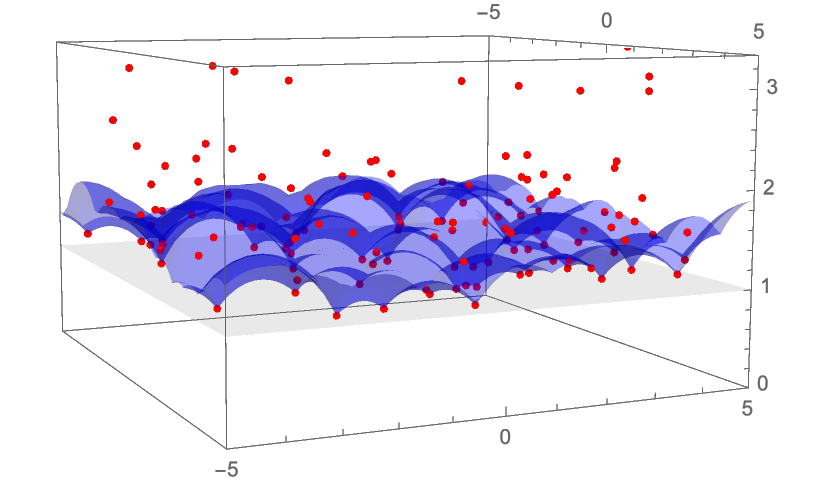}\qquad
	\includegraphics[width=0.3\columnwidth]{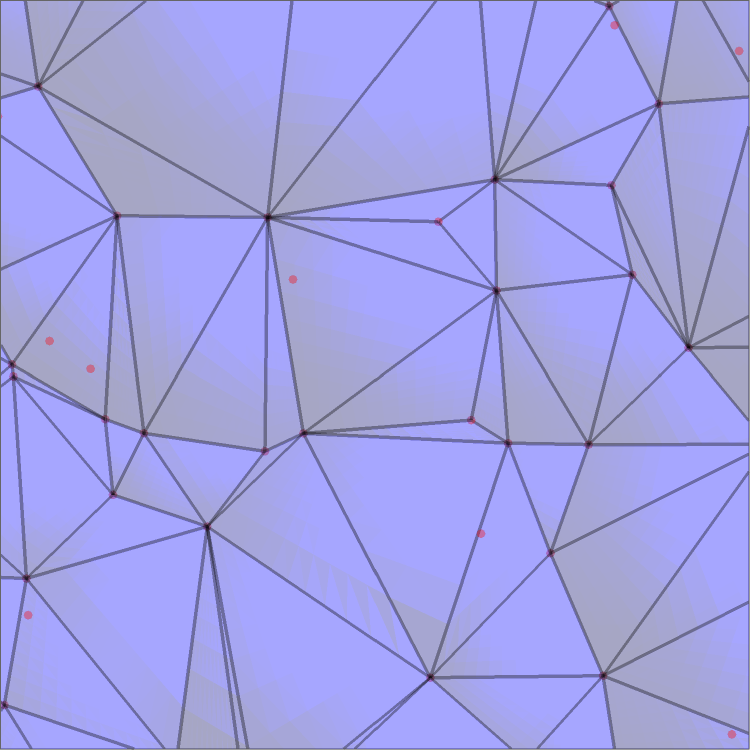}
	\caption{Simulation of a Poisson point process $\eta_{\gamma,\lambda}$ and the Poisson polyhedron $K_{\gamma,\lambda}$ for $d=3$ and $\lambda=1$ in the upper halfspace model. Left: generic view. Right: View from above of the same realization.}
	\label{fig:d=3}
\end{figure}

We denote by $\xi_{d-1}(\gamma,\lambda)$ the cell intensity of the random tessellation $\cD_{\gamma,\lambda}$, which intuitively can be described as the mean number of cells of $\cD_{\gamma,\lambda}$ per unit volume in $\RR^{d-1}$. Since, by Theorem~\ref{thm:DlambdaIsTessellation}, we have $\cD_{\gamma,\lambda}=\mathcal{L}^*(\widetilde{\eta}_{\gamma,\lambda})$, we may formally define $\xi_{d-1}(\gamma,\lambda)$ as follows. For a bounded Borel set $W\subset\RR^{d-1}$ with $0<\vol_{\RR^{d-1}}(W)<\infty$,
\begin{align}
	\xi_{d-1}(\gamma,\lambda) & := \frac{1}{d!\,\vol_{\RR^{d-1}}(W)}\EE\hspace{-0.2cm}\sum_{(\bz_1,\ldots,\bz_d)\in(\widetilde\eta_{\gamma,\lambda})_{\neq}^d}\hspace{-0.4cm}{\bf 1}\{A(\bz_1,\ldots,\bz_d)\in W\}
	\, {\bf 1}\{{\rm int}\,\Pi(\bz_1,\ldots,\bz_d)\cap\widetilde\eta_{\gamma,\lambda}=\varnothing\}.\label{eq:CellIntensity}
\end{align}
This definition does not depend on the choice of $W$. Indeed, since the height density $f_{\gamma,\lambda}$ does not depend on the spatial coordinate, the point process $\widetilde\eta_{\gamma,\lambda}$ and hence the tessellation $\cD_{\gamma,\lambda}$ are stationary with respect to the spatial coordinate, so that the expectation in \eqref{eq:CellIntensity} defines a translation-invariant, locally finite measure as a function of $W$, which is therefore a constant multiple of $\vol_{\RR^{d-1}}$.
In the next theorem we derive an explicit representation for $\xi_{d-1}(\gamma,\lambda)$ in terms of the Gauss hypergeometric function $_2F_1$.

\begin{theorem}\label{thm:CellIntensity}
	For $\gamma,\lambda>0$ we have
	\[
		\xi_{d-1}(\gamma,\lambda) =\frac{\gamma^d\,\lambda^{1-d}}{2d}
		\int_0^1 (1-s)^{\frac{d-3}{2}}
		\exp\!\Big(
		-\gamma\frac{\kappa_{d-1}}{d+1}
		s^{\frac{d+1}{2}}
		{}_2F_1\!\Big(\tfrac{d+1}{2},\tfrac{d+1}{2};\tfrac{d+3}{2};s\Big)
		\Big)\,I_d(s)\,\dint s,
	\]
	where
	\[
		I_d(s) := \int_{(\RR^{d-1})^d}
		\Delta_{d-1}(\bz_1,\ldots,\bz_d)
		\prod_{i=1}^d(1-\|\bz_i\|^2)^{-\frac{d+1}{2}}
		\,{\bf 1}\{\|\bz_i\|^2\le s\}
		\,\dint\bz_i,\qquad s\in(0,1),
	\]
	with $\Delta_{d-1}(\bv_1,\ldots,\bv_d)=\vol_{\RR^{d-1}}([\bv_1,\ldots,\bv_d])$.
\end{theorem}

\begin{remark}\label{rem:Hypergeometric}
	The function ${}_2F_1\!\big(\tfrac{d+1}{2},\tfrac{d+1}{2};\tfrac{d+3}{2};s\big)$ has an alternative more explicit representation, namely
	\[
		{}_2F_1\!\big(\tfrac{d+1}{2},\tfrac{d+1}{2};\tfrac{d+3}{2};s\big)=\frac{2^{\lfloor \frac{d}{2}\rfloor}(d+1)}{(d-1)!!}\cdot\begin{cases}
			\frac{\dint^k}{\dint s^k}\big(-\frac{\log(1-s)}{2s}\big),\qquad              & d=2k+1, \\[0.1cm]
			\frac{\dint^k}{\dint s^k}\big(\frac{\arcsin(\sqrt{s})}{\sqrt{s}}\big),\qquad & d=2k.
		\end{cases}
	\]
	Indeed by \cite[Equation 15.4.1]{NIST} we have ${}_2F_1\!(1,1;2;s)=-\frac{\log(1-s)}{s}$ and by \cite[Equation 15.4.4]{NIST} we get ${}_2F_1\!\big(\frac{1}{2},\frac{1}{2};\frac{3}{2};s\big)=\frac{\arcsin(\sqrt{s})}{\sqrt{s}}$. Now applying
	\[
		\frac{\dint}{\dint s}{}_2F_1\!(a,b;c;s)=\frac{ab}{c}{}_2F_1\!(a+1,b+1;c+1;s),
	\]
	from \cite[Equation 15.5.1]{NIST} iteratively, we obtain the above representation.
\end{remark}

In contrast to the hypergeometric function the integral $I_d(s)$ does not have a simpler representation for general $d$. If $d=2$ we can make the formula for $\xi_{d-1}(\gamma,\lambda)$ in Theorem~\ref{thm:CellIntensity} more explicit. In particular, the geometric integral $I_d(s)$ can in this case be evaluated.

\begin{corollary}\label{cor:CellIntensity_d2}
	Let $d=2$ and $\gamma,\lambda>0$. Then
	\[
		\xi_1(\gamma,\lambda)
		=\frac{\gamma^2}{\lambda}
		\int_0^1 \frac{1}{\sqrt{1-s}}\,
		\Big(\frac{\sqrt{s}}{1-s}-\log\!\Big(\frac{1+\sqrt{s}}{\sqrt{1-s}}\Big)\Big)\,
		\exp\!\Big(
		-2\gamma\Big(\frac{\sqrt{s}}{\sqrt{1-s}}-\arcsin(\sqrt{s})\Big)
		\Big)\,\dint s.
	\]
\end{corollary}

Using the integral representation for $\xi_1(\gamma,\lambda)$ we get the following asymptotics as $\lambda\downarrow 0$.

\begin{corollary}\label{cor:lambda_asymptotics_d2}
	Let $d=2$ and set $\gamma=\lambda$. Then, as $\lambda\downarrow 0$,
	\[
		\xi_1(\lambda,\lambda)
		=1-\pi\lambda+o(\lambda).
	\]
	In particular, $\lim_{\lambda\downarrow 0}\xi_1(\lambda,\lambda)=1$.
\end{corollary}

Our next theorem shows that the limiting phenomenon of Corollary~\ref{cor:lambda_asymptotics_d2} persists for all space dimensions $d\geq 2$, but the argument becomes much more involved.

\begin{theorem}\label{thm:AsymptoticsGeneralD}
	Let $d\geq 2$ and set $\gamma=\lambda^{d-1}$. Then
	\[
		\lim_{\lambda\downarrow0}\xi_{d-1}(\lambda^{d-1},\lambda)
		=
		\frac{2^{d}}{d\,(d-1)^3}\,
		\pi^{\frac{d-2}{2}}\,
		\frac{\Gamma\!\big(\frac{(d-1)^2+1}{2}\big)}{\Gamma\!\big(\frac{(d-1)^2}{2}\big)}\,
		\left(\frac{\Gamma(\frac{d+1}{2})}{\Gamma(\frac d2)}\right)^{\,d-1}=:\xi_{d-1}^{\operatorname{PD}}\Big(\frac{1}{d-1}\Big).
	\]
\end{theorem}

\begin{remark}
	The limit in Theorem~\ref{thm:AsymptoticsGeneralD} is asymptotic to $(2\pi d)^{\frac{d}{2}+O(1)}$ for $d\to \infty$. To see this, we apply the asymptotic relation
	$$
		\frac{\Gamma(x+\frac{1}{2})}{\Gamma(x)} = \sqrt{x}\Big(1-\frac{1}{8x}+O(x^{-2})\Big),\qquad x\to\infty,
	$$
	from \cite[Equation (5.11.13)]{NIST} first with $x=\frac{(d-1)^2}{2}$ and then with $x=\frac{d}{2}$. After simplifications, this yields
	$$
		\xi_{d-1}^{\operatorname{PD}}\Big(\frac{1}{d-1}\Big) = \frac{e^{-1/4}}{\pi\,d^{7/2}}\,(2\pi d)^{d/2}(1+O(d^{-1})) = (2\pi d)^{d/2-7/2+o(1)},\qquad d\to\infty.
	$$
\end{remark}

\begin{remark}
	Theorem~\ref{thm:AsymptoticsGeneralD} concerns the low-intensity regime and
	leads to new geometric phenomena. By contrast, the high-intensity regime is
	expected to recover the Euclidean case studied in
	\cite{CSY13,SY08}. This is the familiar local limit arising in the
	approximation of a smooth convex body near its boundary, and it is natural
	to expect the same behaviour for horoballs, whose boundary horospheres are
	totally umbilic hypersurfaces with constant normal curvature equal to one.

	More precisely, if one sets \(\gamma=\lambda^{d+1}\), then, as
	\(\lambda\to\infty\), the transformed downward-paraboloid process is expected
	to converge to the tessellation induced by the parabolic hull process. We do
	not pursue this limit here. However, at the level of cell intensities,
	Theorem~\ref{thm:CellIntensity} yields
	\[
		\lim_{\lambda\to\infty}
		\xi_{d-1}(\lambda^{d+1},\lambda)
		=
		c_d,
	\]
	where, for example, $c_2=(\frac{2}{3})^{1/3}\Gamma(\frac53)$ and $c_3=\frac{35}{24}$.
\end{remark}

\subsection{Convergence to the Poisson--Delaunay tessellation}

We note that the limit in Theorem~\ref{thm:AsymptoticsGeneralD} coincides with the cell intensity of the classical Poisson--Delaunay tessellation $\cD_{1/(d-1)}$ in $\RR^{d-1}$ of intensity $1/(d-1)$, see \cite[Equation (10.31)]{SW}. This serves as evidence that the whole tessellation $\cD_{\lambda^{d-1},\lambda}$ as $\lambda\downarrow 0$ might converge to $\cD_{1/(d-1)}$. We will next show that this is indeed the case and the above convergence holds on the level of skeletons and on the level of typical cells.

Denote by
\[
	\cS_{\lambda}:={\rm skel}(\cD_{\lambda^{d-1},\lambda})=\bigcup_{C\in \cD_{\lambda^{d-1},\lambda}} \partial C
\]
and by
\[
	\cS:={\rm skel}(\cD_{1/(d-1)})=\bigcup_{C\in \cD_{1/(d-1)}} \partial C,
\]
the skeletons of random tessellations $\cD_{\lambda^{d-1},\lambda}$ and $\cD_{1/(d-1)}$, respectively. Note that $\cS_{\lambda}$ and $\cS$ are random closed sets, namely measurable maps from the probability space to the space of closed subsets of $\RR^{d-1}$ equipped with Fell topology, see \cite{SW}.

\begin{theorem}\label{thm:SkeletonConvergence}
	As $\lambda\downarrow 0$ the distribution of $\cS_{\lambda}$ converges weakly to the distribution of the random closed set $\cS$. The above convergence holds in the stronger sense, namely given a sequence $(\lambda_n)_{n\in\NN}$, $\lambda_n\to 0$ we may define $((\cS_{\lambda_n})_{n\in\NN}, \cS)$ on the same probability space, such that for every $R>0$,
	\[
		\lim_{n\to\infty}\PP(\cS_{\lambda_n}\cap \BB_R^{d-1}\neq \cS\cap \BB_R^{d-1})=0.
	\]
\end{theorem}

\begin{remark}
The low-intensity scaling also identifies the relevant height scale.
Indeed, under $\gamma=\lambda^{d-1}$, the change of variables
$y=\lambda r$ transforms the intensity measure of
$\eta_{\lambda^{d-1},\lambda}$ into $\dint\bv\,r^{-d}\,\dint r$, $r\geq 1$.
Thus its spatial projection is a homogeneous Poisson point process
of intensity $1/(d-1)$, whose points carry independent marks $R$ with $\PP(R>r)=r^{-(d-1)}$, $r\geq 1$.
The hyperbolic distance of $(\bv,\lambda R)$ from the horosphere
$H_\lambda$ is $\log R$, and hence has an exponential distribution
with parameter $d-1$, independently of $\lambda$. In particular, for every bounded Borel set $W\subset\RR^{d-1}$
and every $L>0$,
\[
    \PP\bigl(
        \exists\,(\bv,y)\in\eta_{\lambda^{d-1},\lambda}:
        \bv\in W,\ \log(y/\lambda)>L
    \bigr)
    \leq
    \frac{\vol_{\RR^{d-1}}(W)}{d-1}\,e^{-(d-1)L}.
\]
Consequently, the hyperbolic heights of the input points with
spatial projection in $W$, and therefore also of the extreme
points among them, are bounded in probability, uniformly in
$\lambda$. In this sense, the relevant Euclidean height scale is $\lambda$. 
\end{remark}

Further we note that although the convergence on the level of skeletons has a ``global'' flavour, it is not sensitive enough to provide information about convergence of typical cells, since it does not recognize ``tiny'' cells, see the discussion in \cite{GiWL26}. On the other hand, together with convergence of cell intensities established in Theorem~\ref{thm:AsymptoticsGeneralD} it implies the weak convergence of typical cells as follows by \cite[Proposition 6.3]{GKT24}. Let $Z_{\lambda}$ be the typical cell of $\cD_{\lambda^{d-1},\lambda}$ and $Z^{\operatorname{PD}}_{1/(d-1)}$ be the typical cell of $\cD_{1/(d-1)}$ in the sense of Palm distributions, see \cite[Chapter 10]{SW}.

\begin{corollary}\label{cor:TypicalCellConvergence}
	As $\lambda\downarrow 0$ the distribution of the random simplex $Z_{\lambda}$ converges weakly to that of $Z^{\operatorname{PD}}_{1/(d-1)}$.
\end{corollary}

Theorem~\ref{thm:SkeletonConvergence} implies convergence in probability of all local
statistics which are determined by the tessellation inside a fixed
deterministic window. More precisely, let $R>0$, and let $\Phi$ be a
real-valued functional on tessellations which is $R$-local in the sense
that
\[
	\Phi(\cT_1)=\Phi(\cT_2)
	\qquad\text{whenever}\qquad
	\skel(\cT_1)\cap\BB_R^{d-1}
	=
	\skel(\cT_2)\cap\BB_R^{d-1},
\]
where $\cT_1$ and $\cT_2$ are tessellations of $\RR^{d-1}$. Then,
under the coupling from Theorem~\ref{thm:SkeletonConvergence},
\[
	\PP\left(
	\Phi(\cD_{\lambda^{d-1},\lambda})
	\neq
	\Phi(\cD_{1/(d-1)})
	\right)
	\longrightarrow0,
	\qquad \lambda\downarrow0.
\]
Consequently,
$\Phi(\cD_{\lambda^{d-1},\lambda})$ converges in probability, and hence
also in distribution, to $\Phi(\cD_{1/(d-1)})$. If, in addition,
$\Phi$ is bounded, then the corresponding expectations converge as well.
This applies, for instance, to bounded local functionals of the adjacency
graph and to other statistics that are completely determined by the
tessellation in a fixed observation window. Similarly, the weak convergence of the typical cells in
Corollary~\ref{cor:TypicalCellConvergence} implies convergence of all continuous typical-cell functionals.  For example, the diameter and the volume of $Z_\lambda$ converge in
distribution to the corresponding functionals of the typical
Poisson--Delaunay cell $Z_{1/(d-1)}^{\operatorname{PD}}$. Convergence of the associated expectations, or of higher moments, again requires additional uniform integrability estimates.

On the other hand ,convergence of face intensities follows by combining
Theorem~\ref{thm:AsymptoticsGeneralD} with
Corollary~\ref{cor:TypicalCellConvergence}. The key observation
is that each face intensity equals the cell intensity multiplied
by an expected internal angle sum of the typical cell. Since
these angle sums are bounded and continuous on the space of
nondegenerate simplices, convergence of the typical cells implies
convergence of the corresponding expectations. We thus obtain that, for every
$k\in\{0,\ldots,d-1\}$, the $k$-face intensity of
$\cD_{\lambda^{d-1},\lambda}$ converges to the corresponding face
intensity of the limiting Poisson--Delaunay tessellation. To formally present the result, define
\begin{equation}\label{eqn:k-intensity}
	\xi_k(\gamma,\lambda)
	:=
	\frac{1}{\vol_{\RR^{d-1}}(W)}\;
	\EE
	\sum_{F\in\mathcal F_k(\cD_{\gamma,\lambda})}
	{\bf 1}\{\bz(F)\in W\},
\end{equation}
where $\mathcal F_k(\cD_{\gamma,\lambda})$ denotes the collection of
$k$-dimensional faces of $\cD_{\gamma,\lambda}$, and
$\bz(F)\in F$ is the barycentre of $F\in\mathcal F_k(\cD_{\gamma,\lambda})$, and $W\subset\RR^{d-1}$ is a bounded Borel set with $0<\vol_{\RR^{d-1}}(W)<\infty$. As in \eqref{eq:CellIntensity}, stationarity of $\cD_{\gamma,\lambda}$ implies that this definition does not depend on the choice of $W$. For $k=d-1$ it attributes cells to the window via their barycentre rather than via the apex used in \eqref{eq:CellIntensity}, but, again by stationarity, both conventions yield the same cell intensity $\xi_{d-1}(\gamma,\lambda)$. Analogously, we denote by
$\xi_k^{\operatorname{PD}}(\gamma)$ the $k$-face intensity of the
Poisson--Delaunay tessellation $\cD_\gamma$. The limiting face intensities admit an explicit representation in terms of the
typical Poisson--Delaunay simplex. Let
$Z_\gamma^{\operatorname{PD}}$ denote the typical full-dimensional cell of
$\cD_\gamma$, and let $\sigma_\ell(P)$ be the sum of the internal angles
of a simplex $P$ at all faces with $\ell$ vertices. We normalize $\sigma_\ell(P)$ in such a way that the angle of the whole normal space is equal to one, see \cite[Section 6.1]{GKT22} and \cite[Equation (7.2)]{MT14}. Then
\[
	\xi_k^{\operatorname{PD}}(\gamma)
	=
	\xi_{d-1}^{\operatorname{PD}}(\gamma)\,
	\EE\sigma_{k+1}\bigl(Z_\gamma^{\operatorname{PD}}\bigr),
	\qquad k\in\{0,\ldots,d-1\}.
\]
Moreover,
\[
	\xi_{d-1}^{\operatorname{PD}}(\gamma)
	=
	\frac{1}{\EE\vol_{\RR^{d-1}}\bigl(Z_\gamma^{\operatorname{PD}}\bigr)}\qquad\text{and}\qquad\EE\sigma_{k+1}\bigl(Z_\gamma^{\operatorname{PD}}\bigr)
	=
	\textup{j}_{d,k+1}\Big(-\frac{1}{2}\Big),
\]
where $\textup{j}_{d,k+1}(-\frac{1}{2})$ has the explicit one-dimensional integral
representation given in \cite[Equation~(6.1)]{GKT22}.
For example, in dimension $d=3$, where the limiting tessellation is a planar Poisson--Delaunay triangulation,
\[
	\xi_0^{\operatorname{PD}}(\gamma)=\gamma,\qquad
	\xi_1^{\operatorname{PD}}(\gamma)=3\gamma,\qquad
	\xi_2^{\operatorname{PD}}(\gamma)=2\gamma.
\]

\begin{corollary}\label{cor:AllFaceIntensities}
	For every $k\in\{0,\ldots,d-1\}$,
	\[
		\xi_k(\lambda^{d-1},\lambda)
		\longrightarrow
		\xi_k^{\operatorname{PD}}\Big(\frac{1}{d-1}\Big),\qquad \lambda\downarrow 0.
	\]
\end{corollary}

\begin{remark}\label{rmk:k-intensity_normalized}
For $k\in\{0,\ldots,d-1\}$, the intensity
$\xi_k(\gamma,\lambda)$ is measured per unit Euclidean volume
in the spatial coordinates. The hyperbolic isometry
$(\bv,y)\mapsto(\bv/\lambda,y/\lambda)$ maps
$K_{\gamma,\lambda}$ in distribution onto $K_{\gamma,1}$.
Consequently,
\[
    \xi_k(\gamma,\lambda)
    =\lambda^{-(d-1)}\xi_k(\gamma,1).
\]
Since
\[
    \vol_{\HH^d}(C_\lambda(W))
    =\frac{\vol_{\RR^{d-1}}(W)}{(d-1)\lambda^{d-1}},
\]
we may equivalently normalize the expected projected face count
by the hyperbolic volume of the observation cylinder and define
\[
    \tilde{\xi}_k(\gamma)
    :=(d-1)\lambda^{d-1}\xi_k(\gamma,\lambda)
    =(d-1)\xi_k(\gamma,1).
\]
This quantity is independent of $\lambda$ and of the window $W$.
It is a normalization of projected face counts by cylinder
volume; it does not assert that a point process of face centres
in $\HH^d$ has intensity measure proportional to hyperbolic
volume.
\end{remark}

\subsection{The local volume and surface area}

\begin{figure}[t]
	\centering
	\begin{tikzpicture}[
			scale=1.25,
			>=stealth,
			line cap=round,
			line join=round
		]


		\def\lam{0.75}      
		\def\top{4.30}

		\def\xa{-1.60}
		\def\xb{ 1.60}

		\def\cA{-2.00}
		\def\rA{ 2.00}

		\def\cB{ 0.00}
		\def\rB{ 1.70}

		\def\cC{ 2.00}
		\def\rC{ 2.00}


		\pgfmathsetmacro{\xAB}{
			(\rB*\rB-\rA*\rA+\cA*\cA-\cB*\cB)/(2*(\cA-\cB))
		}

		\pgfmathsetmacro{\xBC}{
			(\rC*\rC-\rB*\rB+\cB*\cB-\cC*\cC)/(2*(\cB-\cC))
		}


		\pgfmathsetmacro{\xAL}{
			\cA-sqrt(\rA*\rA-\lam*\lam)
		}
		\pgfmathsetmacro{\xAR}{
			\cA+sqrt(\rA*\rA-\lam*\lam)
		}

		\pgfmathsetmacro{\xBL}{
			\cB-sqrt(\rB*\rB-\lam*\lam)
		}
		\pgfmathsetmacro{\xBR}{
			\cB+sqrt(\rB*\rB-\lam*\lam)
		}

		\pgfmathsetmacro{\xCL}{
			\cC-sqrt(\rC*\rC-\lam*\lam)
		}
		\pgfmathsetmacro{\xCR}{
			\cC+sqrt(\rC*\rC-\lam*\lam)
		}


		\pgfmathsetmacro{\yAB}{
			sqrt(\rA*\rA-(\xAB-\cA)^2)
		}
		\pgfmathsetmacro{\yBC}{
			sqrt(\rB*\rB-(\xBC-\cB)^2)
		}


		\draw (-3.55,0) -- (4.05,0);


		\fill[blue!2]
		(-3.35,\lam) rectangle (3.85,\top);

		\fill[blue!7]
		(\xa,\lam) rectangle (\xb,\top);

		%

		\begin{scope}
			\clip
			(\xa,\lam) rectangle (\xb,\top);

			\fill[blue!22]
			(\xAL,\lam)
			--
			plot[
					domain=\xAL:\xAR,
					samples=140
				]
			(\x,{sqrt(\rA*\rA-(\x-\cA)^2)})
			--
			(\xAR,\lam)
			-- cycle;

			\fill[blue!22]
			(\xBL,\lam)
			--
			plot[
					domain=\xBL:\xBR,
					samples=140
				]
			(\x,{sqrt(\rB*\rB-(\x-\cB)^2)})
			--
			(\xBR,\lam)
			-- cycle;

			\fill[blue!22]
			(\xCL,\lam)
			--
			plot[
					domain=\xCL:\xCR,
					samples=140
				]
			(\x,{sqrt(\rC*\rC-(\x-\cC)^2)})
			--
			(\xCR,\lam)
			-- cycle;
		\end{scope}


		\draw[blue!32]
		plot[
				domain=\xAL:\xAR,
				samples=140
			]
		(\x,{sqrt(\rA*\rA-(\x-\cA)^2)});

		\draw[blue!32]
		plot[
				domain=\xBL:\xBR,
				samples=140
			]
		(\x,{sqrt(\rB*\rB-(\x-\cB)^2)});

		\draw[blue!32]
		plot[
				domain=\xCL:\xCR,
				samples=140
			]
		(\x,{sqrt(\rC*\rC-(\x-\cC)^2)});

		%

		\draw[blue,very thick]
		plot[
				domain=\xa:\xAB,
				samples=80
			]
		(\x,{sqrt(\rA*\rA-(\x-\cA)^2)});

		\draw[blue,very thick]
		plot[
				domain=\xAB:\xBC,
				samples=100
			]
		(\x,{sqrt(\rB*\rB-(\x-\cB)^2)});

		\draw[blue,very thick]
		plot[
				domain=\xBC:\xb,
				samples=80
			]
		(\x,{sqrt(\rC*\rC-(\x-\cC)^2)});


		\draw[blue,dashed,thick]
		(-3.35,\lam) -- (3.85,\lam);

		\node[
			black,
			left
		]
		at (-3.35,\lam)
		{$H_\lambda$};

		\node[
			black,
			left
		]
		at (-3.35,\top-0.2)
		{$B_\lambda^\infty$};

		\draw[gray!65,densely dashed]
		(\xa,\lam) -- (\xa,\top+0.3);

		\draw[gray!65,densely dashed]
		(\xb,\lam) -- (\xb,\top+0.3);

		\draw[
			decorate,
			decoration={
					brace,
					mirror,
					amplitude=5pt
				}
		]
		(\xa,0) -- (\xb,0)
		node[midway,below=7pt] {$W$};



		\node[black]
		at (0.00,4.62)
		{$C_\lambda(W)$};

		\node[
			blue!75!black,
			align=center
		]
		at (0,3.28)
		{$K_{\gamma,\lambda}\cap C_\lambda(W)$};


	\end{tikzpicture}
	\caption{Local missed volume (blue area) and surface area (blue boundary curve) of $K_{\gamma,\lambda}$ in the horospherical cylinder $C_\lambda(W)$ above $W\subset\RR^{d-1}$.}
	\label{fig:MissedVolume}
\end{figure}
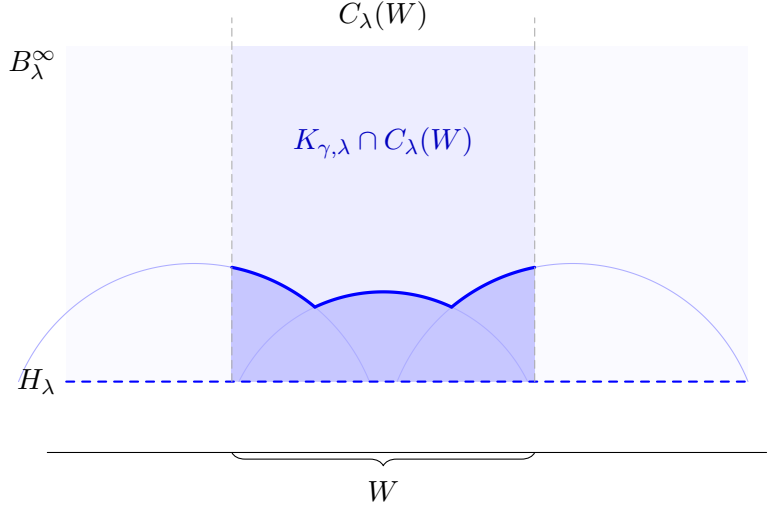

Another classical functional in the Euclidean theory of random polytopes is
the volume of the random polytope. If $K\subset\RR^d$ is a convex body and
$K_n$ is the convex hull of $n$ random points in $K$, one studies
$\vol_{\RR^d}(K_n)$, or equivalently the missed volume
$\vol_{\RR^d}(K\setminus K_n)$. Closely related is the surface area of the boundary of $K_n$. Recently, hyperbolic volumes of Euclidean
beta-polytopes, viewed in the appropriate model of hyperbolic space,
have been studied in \cite{KS2026}.
In the present horoball model the global volume and surface area of
$K_{\gamma,\lambda}$ are not meaningful, since
$K_{\gamma,\lambda}$ is unbounded. However, its one-ended structure
suggests a natural localization in the horospherical direction.

Let $W\subset\RR^{d-1}$ be a bounded Borel set with
$0<\vol_{\RR^{d-1}}(W)<\infty$ and define the
horospherical cylinder $C_\lambda(W):= \{(\bv,y)\in\HH^d:\bv\in W,\ y\ge\lambda\}=W\times[\lambda,\infty)$
together with its horospherical base
$W_\lambda	:= \partial B_\lambda^\infty\cap C_\lambda(W) =W\times\{\lambda\}$.
We recall the definition \eqref{def:local_vol} of the local volume and surface-area ratios $V_{\gamma,\lambda}(W)$ and $S_{\gamma,\lambda}(W)$, also see Figure~\ref{fig:MissedVolume}. Let us note that
$\frac{1}{\lambda} C_{\lambda}(W) = C_{1}\bigl(\frac{1}{\lambda}W\bigr)$ and
$\frac{1}{\lambda} \bigl(W_{\lambda}\bigr) = \bigl(\frac{1}{\lambda}W\bigr)_{1}$,
which yields
\begin{equation*}
	V_{\gamma,\lambda}(W) \overset{d}{=} V_{\gamma,1}\bigl(\tfrac{1}{\lambda}W\bigr) \qquad \text{and}\qquad
	S_{\gamma,\lambda}(W) \overset{d}{=} S_{\gamma,1}\bigl(\tfrac{1}{\lambda}W\bigr).
\end{equation*}

The random variable $V_{\gamma,\lambda}(W)$, which takes values in
$[0,1]$, is the natural local analogue of the volume-ratio functional for
bounded random polytopes in Euclidean spaces. The corresponding normalised
local missed volume is
\[
	\frac{\vol_{\HH^d}\bigl(
		(B_\lambda^\infty\setminus K_{\gamma,\lambda})
		\cap C_\lambda(W)
		\bigr)}
	{\vol_{\HH^d}\bigl(C_\lambda(W)\bigr)}
	=
	1-V_{\gamma,\lambda}(W).
\]
The functional $S_{\gamma,\lambda}(W)$ measures the amount of
hyperbolic boundary area of the random polyhedron per unit hyperbolic
surface area of the horospherical base $W_\lambda$ of the observation
cylinder.

We now give exact formulas for the expectations of
$V_{\gamma,\lambda}(W)$ and $S_{\gamma,\lambda}(W)$. These are the
volume and surface-area analogues of the cell-intensity formula in
Theorem~\ref{thm:CellIntensity}. To state both formulas simultaneously, put for $\ell\in\{0,1\}$,
\[
	F_{\gamma,\lambda}^{(\ell)}(W)
	:=
	\begin{cases}
		V_{\gamma,\lambda}(W), & \ell=0, \\
		S_{\gamma,\lambda}(W), & \ell=1,
	\end{cases}
\]
and define, for $s\in(0,1)$,
\begin{align*}
	J_{d,\ell}(s)
	:=
	\int_{(\RR^{d-1})^d}
	 & \Delta_{d-1}(\bz_1,\ldots,\bz_d)
	\left(
	\int_{[\bz_1,\ldots,\bz_d]}
		(1-\|\bu\|^2)^{-\frac{d-1+\ell}{2}}
	\,\dint\bu
	\right)                             \\
	 & \times\prod_{i=1}^d
	(1-\|\bz_i\|^2)^{-\frac{d+1}{2}}
		{\bf 1}\{\|\bz_i\|^2\le s\}\dint\bz_i.
\end{align*}
For $\ell=0$, the integral in $J_{d,0}(s)$ represents, up to the
factor $d-1$, the hyperbolic volume above the facet, towards $\infty$. For $\ell=1$, the inner integral in $J_{d,1}(s)$ is the
intrinsic hyperbolic $(d-1)$-volume of the corresponding facet in its
supporting totally geodesic hypersurface.

\begin{theorem}\label{thm:LocalVolume}
	Let $W\subset\RR^{d-1}$ be a bounded Borel set with
	$0<\vol_{\RR^{d-1}}(W)<\infty$. Then, for
	$\gamma,\lambda>0$ and $\ell\in\{0,1\}$,
	\[
		\EE F_{\gamma}^{(\ell)}:=\EE F_{\gamma,\lambda}^{(\ell)}(W)
		=
		\frac{\gamma^d}{2d}
		\int_0^1
		(1-s)^{\frac{d-3}{2}}
		\exp\!\Big(
		-\gamma\frac{\kappa_{d-1}}{d+1}
		s^{\frac{d+1}{2}}
		{}_2F_1\!\Big(
		\tfrac{d+1}{2},\tfrac{d+1}{2};
		\tfrac{d+3}{2};s
		\Big)
		\Big)
		J_{d,\ell}(s)\,\dint s.
	\]
	In particular, both expectations $\EE V_{\gamma}=\EE F_{\gamma}^{(0)}=\EE V_{\gamma,\lambda}(W)$ and $\EE S_{\gamma}=\EE F_{\gamma}^{(1)}=\EE S_{\gamma,\lambda}(W)$ are independent of $W$ and
	$\lambda$.
\end{theorem}

The Efron identity for random polytopes in Euclidean space connects the expected vertex number with the expected volume, see \cite[Equation (2.1)]{Reitzner2010} or \cite[Equation (8.12)]{SW} for the classical and \cite[Equation (3.1)]{HHRT2015} or \cite[Equation (15)]{BR:2015} for a Poissonised version. In the same spirit, for our model of random hyperbolic polyhedra we can derive a relationship between the vertex intensity $\xi_0(\gamma,\lambda)$ of $\cD_{\gamma,\lambda}$ and the expected local missed volume.

\begin{corollary}\label{cor:Efron}
	For the vertex intensity we have
	\[
		\lambda^{d-1}\, \xi_0(\gamma,\lambda)
		=\frac{\gamma}{(d-1)}\,
		\bigl(1- \EE V_{\gamma}\bigr).
	\]
\end{corollary}

We give a direct proof of Corollary~\ref{cor:Efron} via the Mecke equation below, but we remark that it can also be derived from results in \cite{LM:2026} on Efron-type identities for stopping sets of Poisson processes, see \cite[Corollary~4.3]{LM:2026} and the treatment of convex hulls of infinite Poisson processes in \cite[Section 7]{LM:2026}. In particular, higher moment identities can be obtained immediately using their framework.

\begin{remark}
	In view of Remark~\ref{rmk:k-intensity_normalized}, Corollary~\ref{cor:Efron} can equivalently be written as $\tilde{\xi}_0(\gamma) = \gamma (1-\EE V_{\gamma})$.
\end{remark}

The asymptotic behaviour for $\gamma\downarrow 0$
of the local volume and surface area can further be analysed by understanding the limiting behaviour of $J_{d,\ell}$ for $s\uparrow 1$.
For this purpose, for ideal points $\bu_1,\dotsc,\bu_d \in \SS^{d-2} \subset \RR^{d-1}$, we consider the ideal hyperbolic simplex $S^{\infty}(\bu_1,\dotsc,\bu_d)$ spanned by these vertices and $\infty$ and denote by $F^\infty(\bu_1,\dotsc,\bu_d)$ the facet of $S^\infty(\bu_1,\dotsc,\bu_d)$ opposite the ideal vertex $\infty$.
For $\ell\in\{0,1\}$, put
\begin{equation}\label{eq:Gdl}
	G_{d,\ell}(\bu_1,\ldots,\bu_d)
	:=
	\begin{cases}
		\vol_{\HH^d}\bigl(
		S^\infty(\bu_1,\ldots,\bu_d)
		\bigr), & \ell=0, \\[1mm]
		\cH_{\HH^d}^{d-1}\bigl(
		F^\infty(\bu_1,\ldots,\bu_d)
		\bigr), & \ell=1,
	\end{cases}
\end{equation}
and define
\[
	C_{d,\ell}
	:=
	\frac{1}{(d-1)^{d-1+\ell}}
	\int_{(\SS^{d-2})^d}
	\Delta_{d-1}(\bu_1,\ldots,\bu_d)\,
	G_{d,\ell}(\bu_1,\ldots,\bu_d)\,
	\sigma(\dint\bu_1)\cdots\sigma(\dint\bu_d).
\]
Since \(\SS^0=\{-1,1\}\), the definition of \(C_{d,\ell}\) immediately
gives
\begin{equation}\label{eq:C20}
	C_{2,0}=4\pi.
\end{equation}
The constant $C_{d,0}$ is finite for every $d\ge2$, whereas
$C_{d,1}$ is finite for $d\ge3$. Indeed, in the latter case
$F^\infty(\bu_1,\ldots,\bu_d)$ is an ideal simplex in
$\HH^{d-1}$ of finite hyperbolic volume. Dimension $d=2$ is
exceptional, since the limiting facet is an ideal geodesic segment and has
infinite hyperbolic length.

To analyse the asymptotic behaviour of the local volume and surface area for $\gamma\downarrow 0$
we use the same
cusp-scaling method as in the proof of
Theorem~\ref{thm:AsymptoticsGeneralD}. Moreover, the Efron-type identity
yields a second-order asymptotic expansion for the vertex intensity, thus
refining Corollary~\ref{cor:AllFaceIntensities} for $k=0$.

\begin{corollary}\label{cor:LocalVolumeAsymptotics}
	Let $\ell\in\{0,1\}$ and assume $d\geq 2$ if $\ell=0$ and $d\geq 3$ if $\ell=1$. Then, as $\gamma\downarrow0$,
	\[
		\EE F^{(\ell)}_{\gamma}
		=
		C_{d,\ell}\,
		\frac{(d-2)!(d-1)^{d-2}}{d\,\kappa_{d-1}^{d-1}}\,
		\gamma
		+
		o(\gamma).
	\]
	In dimension $d=2$, the expected normalised local boundary length satisfies $\EE S_{\gamma}=2\gamma\log\frac{1}{\gamma}+O(\gamma)$.
	In particular,
	\[
		\xi_0(\lambda^{d-1},\lambda)
		=
		\frac{1}{d-1}
		-
		C_{d,0}\,
		\frac{(d-2)!(d-1)^{d-3}}{d\,\kappa_{d-1}^{d-1}}\,
		\lambda^{d-1}
		+
		o(\lambda^{d-1}).
	\]
\end{corollary}

In a last step, we compute the constants $C_{d,\ell}$ from
Corollary~\ref{cor:LocalVolumeAsymptotics} explicitly in dimension $d=3$.

\begin{figure}[t]
	\centering
	\begin{tikzpicture}[
			scale=0.85,
			>=Stealth,
			line cap=round,
			line join=round
		]


		\coordinate (u1) at ( 3.60, 0.55);
		\coordinate (u2) at (-1.85, 1.0-0.04);
		\coordinate (u3) at (-1.10,-0.10-0.12);

		\coordinate (q1) at ($(u1)+(0,2.60)$);
		\coordinate (q2) at ($(u2)+(0,2.60)$);
		\coordinate (q3) at ($(u3)+(0,2.60)$);

		\coordinate (r1) at ($(u1)+(0,5.75)$);
		\coordinate (r2) at ($(u2)+(0,5.75)$);
		\coordinate (r3) at ($(u3)+(0,5.75)$);

		\coordinate (gA) at (-4.20,-0.65);
		\coordinate (gB) at ( 4.30,-0.65);
		\coordinate (gC) at ( 5.40, 1.50);
		\coordinate (gD) at (-3.10, 1.50);

		\coordinate (hA) at ($(gA)+(0,2.60)$);
		\coordinate (hB) at ($(gB)+(0,2.60)$);
		\coordinate (hC) at ($(gC)+(0,2.60)$);
		\coordinate (hD) at ($(gD)+(0,2.60)$);


		\fill[gray!7]
		(gA)--(gB)--(gC)--(gD)--cycle;

		\draw[gray!35]
		(gA)--(gB)--(gC)--(gD)--cycle;

		\draw[gray!70,thick]
		(0.60,0.47) ellipse [x radius=3.00,y radius=0.86];

		\node[gray!50!black] at (2.40,0.1) {$\mathbb S^1$};
		\node[gray!50!black] at (4.55,0.95) {$\RR^2$};

		\draw[blue!60,thin,dashed]
		(u1)--(u2)--(u3)--cycle;

		\draw[gray!55,thin]
		(3.60,0.47) arc [start angle=0, end angle=180, x radius=3.00, y radius=2.7385];


		\pgfmathsetmacro{\XA}{0.8634}  \pgfmathsetmacro{\YA}{0.7559}
		\pgfmathsetmacro{\PA}{-2.7237} \pgfmathsetmacro{\QA}{0.2062} \pgfmathsetmacro{\RA}{2.4415}

		\pgfmathsetmacro{\XB}{-1.4955} \pgfmathsetmacro{\YB}{0.3648}
		\pgfmathsetmacro{\PB}{0.3648}  \pgfmathsetmacro{\QB}{-0.5973} \pgfmathsetmacro{\RB}{1.8332}

		\pgfmathsetmacro{\XC}{1.2282}  \pgfmathsetmacro{\YC}{0.1586}
		\pgfmathsetmacro{\PC}{2.3589}  \pgfmathsetmacro{\QC}{0.3911} \pgfmathsetmacro{\RC}{2.3616}

		\pgfmathsetmacro{\XS}{0.60} \pgfmathsetmacro{\YS}{0.47}
		\pgfmathsetmacro{\PS}{3.00} \pgfmathsetmacro{\RS}{2.7385}

		\fill[blue!18,opacity=0.34]
		plot[domain=180:0,samples=40,variable=\t] ({\XC+\PC*cos(\t)},{\YC+\QC*cos(\t)+\RC*sin(\t)})
		-- plot[domain=0:139.107,samples=40,variable=\t] ({\XS+\PS*cos(\t)},{\YS+\RS*sin(\t)})
		-- plot[domain=118.173:0,samples=40,variable=\t] ({\XB+\PB*cos(\t)},{\YB+\QB*cos(\t)+\RB*sin(\t)})
		-- cycle;

		\draw[blue,very thick]
		plot[domain=180:0,samples=40,variable=\t] ({\XA+\PA*cos(\t)},{\YA+\QA*cos(\t)+\RA*sin(\t)});

		\draw[blue,very thick]
		plot[domain=180:0,samples=40,variable=\t] ({\XB+\PB*cos(\t)},{\YB+\QB*cos(\t)+\RB*sin(\t)});

		\draw[blue,very thick]
		plot[domain=180:0,samples=40,variable=\t] ({\XC+\PC*cos(\t)},{\YC+\QC*cos(\t)+\RC*sin(\t)});


		\fill[blue!8,opacity=0.15]
		(q1)--(q2)--(u2)
		plot[domain=0:180,samples=40,variable=\t] ({\XA+\PA*cos(\t)},{\YA+\QA*cos(\t)+\RA*sin(\t)})
		--cycle;

		\fill[blue!10,opacity=0.13]
		(q2)--(q3)--(u3)
		plot[domain=0:180,samples=40,variable=\t] ({\XB+\PB*cos(\t)},{\YB+\QB*cos(\t)+\RB*sin(\t)})
		--cycle;

		\fill[blue!12,opacity=0.11]
		(q3)--(q1)--(u1)
		plot[domain=0:180,samples=40,variable=\t] ({\XC+\PC*cos(\t)},{\YC+\QC*cos(\t)+\RC*sin(\t)})
		--cycle;



		\fill[gray!38,opacity=0.38]
		(hA)--(hB)--(hC)--(hD)--cycle;

		\draw[gray!55!black]
		(hA)--(hB)--(hC)--(hD)--cycle;

		\node[gray!50!black] at (4.35,3.45)
		{$H_1$};


		\fill[blue!28,opacity=0.55]
		(q1)--(q2)--(q3)--cycle;

		\draw[blue,very thick]
		(q1)--(q2)--(q3)--cycle;



		\filldraw[blue] (u1) circle (2.1pt);
		\filldraw[blue] (u2) circle (2.1pt);
		\filldraw[blue] (u3) circle (2.1pt);

		\node[right]      at (u1) {$\bu_1$};
		\node[above left] at (u2) {$\bu_2$};
		\node[below left] at (u3) {$\bu_3$};

		\filldraw[blue] (q1) circle (1.8pt);
		\filldraw[blue] (q2) circle (1.8pt);
		\filldraw[blue] (q3) circle (1.8pt);


		\draw[blue,very thick,->] (u1)--(r1);
		\draw[blue,very thick,->] (u2)--(r2);
		\draw[blue,very thick,->] (u3)--(r3);

	\end{tikzpicture}

	\caption{The ideal simplex
		$S^\infty(\bu_1,\bu_2,\bu_3)$ in the upper halfspace model for $\HH^3$.
		The three geodesic edges with ideal endpoint $\infty$ are vertical
		parallel rays. The fourth face is the part of the Euclidean hemisphere
		over $\SS^1$ lying above the flat triangle $[\bu_1,\bu_2,\bu_3]$
		(dashed on $\RR^2$). It meets $\RR^2$ orthogonally at
		$\bu_1,\bu_2,\bu_3$ and touches $H_1$ tangentially. The horosphere
		$H_1$ intersects the simplex in the
		Euclidean triangle with vertices
		$(\bu_1,1),(\bu_2,1),(\bu_3,1)$.
	}
	\label{fig:Sinfinity}
\end{figure}
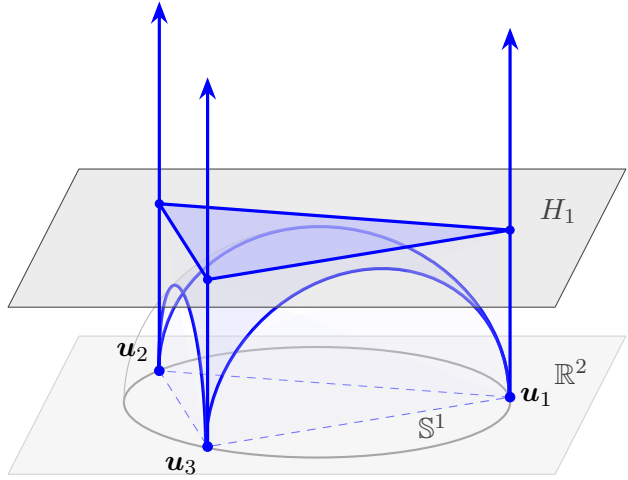

\begin{proposition}\label{cor:Vold=2d=3}
	We have
	\[
		C_{3,0}=\frac{3\pi^3}{4}
		\qquad\text{and}\qquad
		C_{3,1}=\frac{3\pi^3}{2}.
	\]
\end{proposition}

\begin{remark}
	The integrands appearing in the definitions of $C_{d,0}$ and
	$C_{d,1}$ admit intrinsic hyperbolic interpretations. Let $H_1:=\{(\bv,1):\bv\in\RR^{d-1}\}$
	be the horosphere at height $1$. Then
	\[
		\Delta_{d-1}(\bu_1,\ldots,\bu_d)
		=
		\cH^{d-1}_{\HH^d}\bigl(
		S^\infty(\bu_1,\ldots,\bu_d)\cap H_1
		\bigr).
	\]
	Consequently, for $\ell\in\{0,1\}$, the product $\Delta_{d-1}(\bu_1,\ldots,\bu_d)\,G_{d,\ell}(\bu_1,\ldots,\bu_d)$
	can be written as
	\[
		\cH^{d-1}_{\HH^d}\bigl(
		S^\infty(\bu_1,\ldots,\bu_d)\cap H_1
		\bigr)
		\cdot
		\begin{cases}
			\vol_{\HH^d}\bigl(
			S^\infty(\bu_1,\ldots,\bu_d)
			\bigr), & \ell=0, \\[1mm]
			\cH_{\HH^d}^{d-1}\bigl(
			F^\infty(\bu_1,\ldots,\bu_d)
			\bigr), & \ell=1.
		\end{cases}
	\]
	Thus $C_{d,0}$ involves the product of the horospherical
	cross-sectional volume and the volume of the ideal simplex, whereas
	$C_{d,1}$ involves the same cross-sectional volume and the intrinsic
	hyperbolic volume of the ideal facet opposite $\infty$, see
	Figure~\ref{fig:Sinfinity}.
\end{remark}

\section{Proofs}

\subsection{Proofs of Theorem~\ref{thm:UnboundedOneEnded} and Theorem~\ref{thm:DlambdaIsTessellation}}

We first prove Theorem~\ref{thm:DlambdaIsTessellation}, since the result will be used in the proof of Theorem \ref{thm:UnboundedOneEnded}.

\begin{proof}[Proof of Theorem~\ref{thm:DlambdaIsTessellation}]
	By the mapping theorem for Poisson point processes, the image process
	$\widetilde\eta_{\gamma,\lambda}=T(\eta_{\gamma,\lambda})$ is a Poisson
	point process. Its intensity measure is the image of
	$\gamma\vol_{\HH^d}|_{B_\lambda^\infty}$ under $T$, and is therefore given by
	\eqref{eq:IntensityEta}.

	It remains to identify the simplices. Let $(\bv_1,y_1),\ldots,(\bv_d,y_d)\in(\eta_{\gamma,\lambda})_{\neq}^d$, and put
	$h_i=y_i^2$. Let
	$B(\bv_1,y_1,\ldots,\bv_d,y_d)$ be the closed halfball whose boundary is orthogonal to
	$\RR^{d-1}\times\{0\}$ and contains
	$(\bv_1,y_1), \ldots, (\bv_d,y_d)$. Then there are unique
	$\bw\in\RR^{d-1}$ and $r>0$ such that
	\[
		B(\bv_1,y_1,\ldots,\bv_d,y_d)
		=
		\{(\bv,y)\in\RR^{d-1}\times(0,\infty):
		\|\bv-\bw\|^2+y^2\le r^2\}.
	\]
	Applying $T(\bv,y)=(\bv,y^2)$, we obtain
	\[
		T\big(B(\bv_1,y_1,\ldots,\bv_d,y_d)\big)
		=
		\{(\bv,h)\in\RR^{d-1}\times(0,\infty):
		\|\bv-\bw\|^2+h\le r^2\}.
	\]
	This is precisely the intersection of
	$\RR^{d-1}\times(0,\infty)$ with the downward paraboloid
	\[
		\Pi\big((\bv_1,h_1),\ldots,(\bv_d,h_d)\big)
		=
		\{(\bv,h)\in\RR^{d-1}\times\RR:
		h\le r^2-\|\bv-\bw\|^2\}.
	\]
	In particular, the apex of this paraboloid is $(\bw,r^2)$, so that its
	spatial apex coordinate agrees with the spatial apex coordinate of the
	halfball.

	Since $T$ is a bijection from
	$\RR^{d-1}\times(0,\infty)$ onto itself, the empty halfball condition
	\[
		\operatorname{int}B(\bv_1,y_1,\ldots,\bv_d,y_d)
		\cap\eta_{\gamma,\lambda}
		=
		\varnothing
	\]
	is equivalent to the empty paraboloid condition
	\[
		\operatorname{int}\Pi\big((\bv_1,h_1),\ldots,(\bv_d,h_d)\big)
		\cap\widetilde\eta_{\gamma,\lambda}
		=
		\varnothing .
	\]
	Moreover, the transformation $T$ leaves the spatial coordinates
	unchanged. Hence both constructions produce the same simplex
	$[\bv_1,\ldots,\bv_d]$. This proves that $\mathcal D_{\gamma,\lambda}=\mathcal L^*(\widetilde\eta_{\gamma,\lambda})$
	almost surely.
\end{proof}

\begin{proof}[Proof of Theorem~\ref{thm:UnboundedOneEnded}] 
	We first prove the assertion about the ideal boundary. Since $B_\lambda^\infty$
	is a horoball, it is geodesically convex. Hence $K_{\gamma,\lambda}=\conv(\eta_{\gamma,\lambda})\subseteq B_\lambda^\infty$.
	In the upper halfspace compactification, the closure of $B_\lambda^\infty$
	meets the ideal boundary in exactly one point, namely $\infty$.
	Consequently, $\operatorname{cl}(K_{\gamma,\lambda})\cap\partial_\infty\HH^d\subseteq \{\infty\}$, and it remains to show that $\infty$ is actually attained as an ideal
	boundary point. Fix $H>\lambda$ and put $A_H:=\{(\bv,y)\in\HH^d:\lambda\le y\le H\}$.
	Then
	\[
		\vol_{\HH^d}(A_H)
		=
		\int_\lambda^H\int_{\RR^{d-1}} y^{-d}\,\dint \bv\,\dint y
		=
		\infty .
	\]
	Hence $\eta_{\gamma,\lambda}(A_H)=\infty$ almost surely. Since
	$\eta_{\gamma,\lambda}$ is locally finite, there is almost surely a
	sequence of points $(\bv_n,y_n)\in\eta_{\gamma,\lambda}\cap A_H$
	with $\|\bv_n\|\to\infty$. Such a sequence converges to the ideal point
	$\infty$ in the upper halfspace compactification. Since
	$\eta_{\gamma,\lambda}\subseteq K_{\gamma,\lambda}$, this yields $\infty\in\operatorname{cl}(K_{\gamma,\lambda})\cap\partial_\infty\HH^d$ and therefore
	\[
		\operatorname{cl}(K_{\gamma,\lambda})\cap\partial_\infty\HH^d
		=
		\{\infty\}.
	\]
	In particular, $K_{\gamma,\lambda}$ is unbounded and one-ended.

Write $\widehat{K}$ for the closure of $K_{\gamma,\lambda}$
in $\HH^d$. Since $\widehat{K}$ is closed and convex and
has $\infty$ in its ideal closure, it contains the vertical
geodesic ray from each of its points towards $\infty$.
Indeed, these rays are limits of geodesic segments joining
the given point to process points converging to $\infty$.

By Theorem~\ref{thm:DlambdaIsTessellation}, the projected
simplices form a locally finite tessellation
$\cD_{\gamma,\lambda}$ of $\RR^{d-1}$. Each of its cells lifts to
a geodesic simplex with vertices in $\eta_{\gamma,\lambda}$
on a supporting hemisphere whose open lower side contains
no process points. Hence the lifted simplex belongs to
$K_{\gamma,\lambda}$, and $\widehat{K}$ lies on the closed
upper side of the hemisphere.

Every $\bv\in\RR^{d-1}$ belongs to at least one cell of $\cD_{\gamma,\lambda}$.
Let $(\bv,y)\in\RR^{d-1}\times(0,\infty)$ be the corresponding point of its lifted
simplex. The supporting halfspace and the vertical ray
property imply
\[
    \widehat{K}\cap\bigl(\{\bv\}\times(0,\infty)\bigr)
    =\{\bv\}\times[y,\infty).
\]
In particular, lifts of adjacent cells agree on their
common faces. By local finiteness of $\cD_{\gamma,\lambda}$, these lifts form a continuous
graph describing the whole boundary of $\widehat{K}$.
Since every lifted simplex belongs to $K_{\gamma,\lambda}$,
we conclude that $\partial\widehat{K}\subseteq K_{\gamma,\lambda}$.

The Poisson process almost surely contains $d+1$ points
in general position, so $K_{\gamma,\lambda}$ is
full-dimensional. A full-dimensional convex set and its
closure have the same interior. Together with the boundary
inclusion above, this proves
$K_{\gamma,\lambda}=\widehat{K}$, and hence closedness of $K_{\gamma,\lambda}$.

Almost surely, no $d+1$ process points lie on a common
totally geodesic hypersurface. Therefore, the lifted
simplices are precisely the facets of $K_{\gamma,\lambda}$.
Every compact subset of $\HH^d$ has compact spatial
projection and hence meets almost surely only finitely many lifted
simplices and their faces. Thus $K_{\gamma,\lambda}$ is
locally polyhedral with locally finite face structure.
Finally, the locally finite tessellation
$\cD_{\gamma,\lambda}$ has countably infinitely many cells,
since its bounded simplices cover $\RR^{d-1}$.
The facet correspondence proves that the face structure
of $K_{\gamma,\lambda}$ is countably infinite.
\end{proof}

\subsection{Proof of Theorem~\ref{thm:CellIntensity} and its corollaries}

\begin{proof}[Proof of Theorem~\ref{thm:CellIntensity}]
	We start by using the definition \eqref{eq:CellIntensity} of $\xi_{d-1}(\gamma,\lambda)$, in which we may choose $W=[0,1]^{d-1}$, and the multivariate Mecke equation \cite[Theorem 4.4]{LPbook} for Poisson point processes. Together with \eqref{eq:IntensityEta} we get
	\begin{align*}
		\xi_{d-1}(\gamma,\lambda) & = \frac{\gamma^d}{2^dd!}\int_{(\RR^{d-1})^d}\int_{(0,\infty)^d}{\bf 1}\{A(\bv_1,h_1,\ldots,\bv_d,h_d)\in[0,1]^{d-1}\}                                                                                           \\
		                          & \hspace{0.5cm}\times\PP\big[{\rm int}\,\Pi(\bv_1,h_1,\ldots,\bv_d,h_d)\cap\widetilde\eta_{\gamma,\lambda}=\varnothing\big]\prod_{i=1}^d h_i^{-\frac{d+1}{2}}{\bf 1}\{h_i\geq\lambda^2\}\,\dint h_i\,\dint\bv_i.
	\end{align*}
	Next, we consider the change of variables
	\begin{align*}
		G:\RR^{d-1}\times(0,\infty)\times(\inter\BB^{d-1})^d & \to(\RR^{d-1}\times(0,\infty))^d                                              \\
		(\bw,t,\bz_1,\ldots,\bz_d)                     & \mapsto (\bw+t\bz_1,t^2(1-\|\bz_1\|^2),\ldots,\bw+t\bz_d,t^2(1-\|\bz_d\|^2)),
	\end{align*}
	which may be regarded as a parabolic Blaschke--Petkantschin-type integral-geometric transformation as described in \cite[Chapter 7]{SW}.
	The Jacobian $J(G)$ of $G$ is given by
	\[
		J(G)=2^dt^{d^2}(d-1)!\Delta_{d-1}(\bz_1,\ldots,\bz_d),
	\]
	where we recall that $\Delta_{d-1}(\bz_1,\ldots,\bz_d)$ denotes the Euclidean $(d-1)$-volume of the simplex with vertices $\bz_1,\ldots,\bz_d$. This was shown in \cite[Proof of Theorem 1, page 1272]{GKT22}. With this change of variables the paraboloid $\Pi(\bv_1,h_1,\ldots,\bv_d,h_d)$ is the paraboloid with apex $(\bw,t^2)$, which we denote by $\Pi_{(\bw,t^2)}$, while $(\bz_1,\ldots,\bz_d)$ encode the configuration of $(\bv_1,\ldots,\bv_d)$ after shifting them simultaneously by $\bw$ and rescaling by $t$. Since the intersection of $\Pi(\bv_1,h_1,\ldots,\bv_d,h_d)$ with $\RR^{d-1}$ is a ball with centre $\bw$ and radius $t$ this implies that $\bz_i$ belong to the unit ball in $\RR^{d-1}$ for any $1\leq i\leq d$. This allows us to ``decouple'' position $\bw$ and scaling $t$ parameters and configuration $(\bz_1,\ldots,\bz_d)$, reflecting relative positions of $\bv_1,\ldots,\bv_d$ with respect to each other.

	Applying the transformation $G$ to the integral representation for $\xi_{d-1}(\gamma,\lambda)$ we obtain
	\begin{align*}
		\xi_{d-1}(\gamma,\lambda) & = \frac{\gamma^d(d-1)!}{d!}\,\int_{\RR^{d-1}}{\bf 1}\{\bw\in[0,1]^{d-1}\}\int_{(\RR^{d-1})^d}\int_\lambda^\infty t^{d^2}\,\PP\big[{\rm int}\,\Pi_{(\bw,t^2)}\cap\widetilde\eta_{\gamma,\lambda}=\varnothing\big] \\
		                          & \qquad \times \Delta_{d-1}(\bz_1,\ldots,\bz_d)\prod_{i=1}^d(t^2(1-\|\bz_i\|^2))^{-\frac{d+1}{2}}\,{\bf 1}\Big\{1-\|\bz_i\|^2\geq\frac{\lambda^2}{t^2}\Big\}\,\dint\bz_i\,\dint t\,\dint \bw                      \\
		                          & =\frac{\gamma^d}{d}\int_\lambda^\infty t^{-d}\int_{(\RR^{d-1})^d}\PP\big[{\rm int}\,\Pi_{(o,t^2)}\cap\widetilde\eta_{\gamma,\lambda}=\varnothing\big]\,\Delta_{d-1}(\bz_1,\ldots,\bz_d)                          \\
		                          & \qquad \times \prod_{i=1}^d(1-\|\bz_i\|^2)^{-\frac{d+1}{2}}\,{\bf 1}\Big\{1-\|\bz_i\|^2\geq\frac{\lambda^2}{t^2}\Big\}\,\dint\bz_i\,\dint t,
	\end{align*}
	where $o$ is the origin of $\RR^{d-1}$ and in the second step we used that $\widetilde\eta_{\gamma,\lambda}$ is stationary with respect to the spatial coordinate.

	In the next step we determine the probability $\PP[{\rm int}\,\Pi_{(o,t^2)}\cap\widetilde\eta_{\gamma,\lambda}=\varnothing]$, which equals one if $t<\lambda$, since $\widetilde\eta_{\gamma,\lambda}$ has no points of height below $\lambda^2$. The restriction $t\ge\lambda$ above stems from the indicators ${\bf 1}\{1-\|\bz_i\|^2\ge\lambda^2/t^2\}$, which vanish for $t<\lambda$ because $\|\bz_i\|\ge 0$. Using that $\widetilde\eta_{\gamma,\lambda}$ is a Poisson point process we first obtain
	$$
		\PP\big[{\rm int}\,\Pi_{(o,t^2)}\cap\widetilde\eta_{\gamma,\lambda}=\varnothing\big] = \exp\big(-\Lambda_{\gamma,\lambda}(\Pi_{(o,t^2)})\big).
	$$
	Moreover, using \eqref{eq:IntensityEta} we obtain
	\begin{align*}
		\Lambda_{\gamma,\lambda}(\Pi_{(o,t^2)})
		 & =
		\frac{\gamma}{2}
		\int_{\RR^{d-1}}
		\int_{\lambda^2}^{\infty}
		{\bf 1}\{\|\bv\|^2+h\le t^2\}
		h^{-\frac{d+1}{2}}
		\,\dint h\,\dint\bv =
		\frac{\gamma\kappa_{d-1}}{2}
		\int_{\lambda^2}^{t^2}
		(t^2-h)^{\frac{d-1}{2}}
		h^{-\frac{d+1}{2}}
		\,\dint h .
	\end{align*}
	For \(t\ge\lambda\), we use the substitution $h=t^2-\bigl(t^2-\lambda^2\bigr)y$, $y\in[0,1]$.
	Then, by \cite[Equation~15.6.1]{NIST},
	\[
		\Lambda_{\gamma,\lambda}(\Pi_{(o,t^2)})
		=
		\gamma\frac{\kappa_{d-1}}{d+1}
		\Big(1-\frac{\lambda^2}{t^2}\Big)^{\frac{d+1}{2}}
		{}_2F_1\!\Big(
		\tfrac{d+1}{2},\tfrac{d+1}{2};
		\tfrac{d+3}{2};
		1-\frac{\lambda^2}{t^2}
		\Big).
	\]
	As a result,
	$$
		\PP\big[{\rm int}\,\Pi_{(o,t^2)}\cap\widetilde\eta_{\gamma,\lambda}=\varnothing\big] = \exp\!\Big(
		-\gamma\frac{\kappa_{d-1}}{d+1}
		\Big(1-\frac{\lambda^2}{t^2}\Big)^{\frac{d+1}{2}}
		{}_2F_1\!\Big(\tfrac{d+1}{2},\tfrac{d+1}{2};\tfrac{d+3}{2};1-\tfrac{\lambda^2}{t^2}\Big)
		\Big),
	$$
	and the cell intensity simplifies to
	\begin{align*}
		\xi_{d-1}(\gamma,\lambda) & = \frac{\gamma^d}{d}\int_\lambda^\infty t^{-d}\exp\!\Big(
		-\gamma\frac{\kappa_{d-1}}{d+1}
		\Big(1-\frac{\lambda^2}{t^2}\Big)^{\frac{d+1}{2}}
		{}_2F_1\!\Big(\tfrac{d+1}{2},\tfrac{d+1}{2};\tfrac{d+3}{2};1-\tfrac{\lambda^2}{t^2}\Big)
		\Big)                                                                                                                                                                                                                        \\
		                          & \qquad\times \int_{(\RR^{d-1})^d}\Delta_{d-1}(\bz_1,\ldots,\bz_d)\prod_{i=1}^d(1-\|\bz_i\|^2)^{-\frac{d+1}{2}}\,{\bf 1}\Big\{1-\|\bz_i\|^2\geq\frac{\lambda^2}{t^2}\Big\}\,\dint\bz_i\, \dint t.
	\end{align*}

	Finally, we perform the substitution $s=1-\frac{\lambda^2}{t^2}$, $t=\frac{\lambda}{\sqrt{1-s}}$ with $s\in(0,1)$.
	Then $\dint t=\frac{\lambda}{2}(1-s)^{-3/2}\,\dint s$ and $t^{-d}=\lambda^{-d}(1-s)^{d/2}$,
	so that $t^{-d}\,\dint t=\frac{\lambda^{1-d}}{2}(1-s)^{\frac{d-3}{2}}\,\dint s$.
	Moreover,
	\[
		{\bf 1}\Big\{1-\|\bz_i\|^2\ge \frac{\lambda^2}{t^2}\Big\}
		={\bf 1}\big\{\|\bz_i\|^2\le s\big\}.
	\]
	Substitution into the expression for $\xi_{d-1}(\gamma,\lambda)$ yields
	\begin{align*}
		\xi_{d-1}(\gamma,\lambda)
		 & =\frac{\gamma^d\,\lambda^{1-d}}{2d}
		\int_0^1 (1-s)^{\frac{d-3}{2}}
		\exp\!\Big(
		-\gamma\frac{\kappa_{d-1}}{d+1}
		s^{\frac{d+1}{2}}
		{}_2F_1\!\Big(\tfrac{d+1}{2},\tfrac{d+1}{2};\tfrac{d+3}{2};s\Big)
		\Big)\,I_d(s)\,\dint s
	\end{align*}
	with
	$$
		I_d(s) = 	\int_{(\RR^{d-1})^d}
		\Delta_{d-1}(\bz_1,\ldots,\bz_d)
		\prod_{i=1}^d(1-\|\bz_i\|^2)^{-\frac{d+1}{2}}
		\,{\bf 1}\{\|\bz_i\|^2\le s\}
		\,\dint\bz_i.
	$$
	This completes the proof of the theorem.
\end{proof}

\begin{proof}[Proof of Corollary~\ref{cor:CellIntensity_d2}]
	We start from Theorem~\ref{thm:CellIntensity} and set $d=2$. Using $\kappa_1=2$,
	\[
		\xi_1(\gamma,\lambda)
		=\frac{\gamma^2}{4\lambda}
		\int_0^1 (1-s)^{-1/2}\,
		\exp\!\Big(
		-\frac{2\gamma}{3}\,
		s^{3/2}\,
		{}_2F_1\!\Big(\tfrac32,\tfrac32;\tfrac52;s\Big)
		\Big)\,I_2(s)\,\dint s,
	\]
	where
	\[
		I_2(s)=\int_{[-\sqrt{s},\sqrt{s}]^2}|z_1-z_2|\,(1-z_1^2)^{-3/2}(1-z_2^2)^{-3/2}\,\dint z_1\dint z_2.
	\]
	We first eliminate the hypergeometric term. By Remark~\ref{rem:Hypergeometric} we have
	\begin{align*}
		s^{3/2}{}_2F_1\!\Big(\tfrac32,\tfrac32;\tfrac52;s\Big)
		 & =6s^{3/2}\Big(\frac{\arcsin(\sqrt{s})}{\sqrt{s}}\Big)'=3\Big(\frac{\sqrt{s}}{\sqrt{1-s}}-\arcsin(\sqrt{s})\Big),
	\end{align*}
	and hence
	\[
		\exp\!\Big(
		-\frac{2\gamma}{3}\,
		s^{3/2}\,
		{}_2F_1\!\Big(\tfrac32,\tfrac32;\tfrac52;s\Big)
		\Big)
		=
		\exp\!\Big(
		-2\gamma\Big(\frac{\sqrt{s}}{\sqrt{1-s}}-\arcsin(\sqrt{s})\Big)
		\Big).
	\]
	Next we compute $I_2(s)$. Substitute $u_i=z_i/\sqrt{1-z_i^2}$, $i\in\{1,2\}$ (hence, $z_i=u_i/\sqrt{1+u_i^2}$, $\dint u_i=(1-z_i^2)^{-3/2}\dint z_i$) in order to obtain
	\[
		I_2(s)=\int_{[-U,U]^2}\Big|\frac{u_1}{\sqrt{1+u_1^2}}-\frac{u_2}{\sqrt{1+u_2^2}}\Big|\,\dint u_1\dint u_2,
	\]
	where $U=\sqrt{s}/\sqrt{1-s}$.
	Since $u\mapsto u/\sqrt{1+u^2}$ is odd and strictly increasing, symmetry yields
	\begin{align*}
		I_2(s) & =2\int_{-U}^U\int_{-U}^{u_1}\Big(\frac{u_1}{\sqrt{1+u_1^2}}-\frac{u_2}{\sqrt{1+u_2^2}}\Big)\dint u_2\dint u_1 \\
		       & =2\int_{-U}^U\Big((u+U)\frac{u}{\sqrt{1+u^2}}-(U-u)\frac{u}{\sqrt{1+u^2}}\Big)\dint u                         \\
		       & =8\int_0^U \frac{u^2}{\sqrt{1+u^2}}\,\dint u,
	\end{align*}
	where in the second step we integrated the first summand with respect to $u_2$ and set $u_1=u$, while the second summand we integrated with respect to $u_1$ and set $u_2=u$. Using
	\[
		\int \frac{u^2}{\sqrt{1+u^2}}\,\dint u
		=\frac12\Big(u\sqrt{1+u^2}-\operatorname{arsinh}(u)\Big),
	\]
	we obtain
	\[
		I_2(s)=4\Big(U\sqrt{1+U^2}-\operatorname{arsinh}(U)\Big).
	\]
	Now 
	\[
		U\sqrt{1+U^2}=\frac{\sqrt{s}}{1-s}
        \qquad\text{and}\qquad
        \operatorname{arsinh}(U)=\log\!\big(U+\sqrt{1+U^2}\big)
		=\log\!\Big(\frac{1+\sqrt{s}}{\sqrt{1-s}}\Big),
	\]
	so that
	\[
		I_2(s)=4\Big(\frac{\sqrt{s}}{1-s}-\log\!\Big(\frac{1+\sqrt{s}}{\sqrt{1-s}}\Big)\Big).
	\]

	\medskip

	Substituting these identities into the formula from Theorem~\ref{thm:CellIntensity}
	yields
	\[
		\xi_1(\gamma,\lambda)
		=\frac{\gamma^2}{\lambda}
		\int_0^1 (1-s)^{-1/2}\,
		\Big(\frac{\sqrt{s}}{1-s}-\log\!\Big(\frac{1+\sqrt{s}}{\sqrt{1-s}}\Big)\Big)\,
		\exp\!\Big(
		-2\gamma\Big(\frac{\sqrt{s}}{\sqrt{1-s}}-\arcsin(\sqrt{s})\Big)
		\Big)\,\dint s,
	\]
	which simplifies to the claimed expression.
\end{proof}

\begin{proof}[Proof of Corollary~\ref{cor:lambda_asymptotics_d2}]
Setting $\gamma=\lambda$ in Corollary~\ref{cor:CellIntensity_d2}
and substituting $s=t^2/(1+t^2)$, we obtain
\[
    \xi_1(\lambda,\lambda)
    =2\lambda\int_0^\infty
    \left(
        \frac{t^2}{1+t^2}
        -\frac{t\,\operatorname{arsinh}(t)}{(1+t^2)^{3/2}}
    \right)e^{-2\lambda\psi(t)}\,\dint t,
\]
where $\psi(t):=t-\arctan(t)$. Since
$\psi'(t)=t^2/(1+t^2)$, $\psi(0)=0$ and
$\psi(t)\to\infty$ as $t\to\infty$, this gives
\[
    \xi_1(\lambda,\lambda)
    =1-2\lambda\int_0^\infty
    \frac{t\,\operatorname{arsinh}(t)}{(1+t^2)^{3/2}}
    e^{-2\lambda\psi(t)}\,\dint t.
\]
Integration by parts yields
\[
    \int_0^\infty
    \frac{t\,\operatorname{arsinh}(t)}{(1+t^2)^{3/2}}\,\dint t
    =
    \left[-\frac{\operatorname{arsinh}(t)}{\sqrt{1+t^2}}
    \right]_0^\infty
    +\int_0^\infty\frac{\dint t}{1+t^2}
    =\frac{\pi}{2}.
\]
As $\psi\geq0$, dominated convergence therefore implies
\[
    \xi_1(\lambda,\lambda)
    =1-\pi\lambda+o(\lambda),
    \qquad \lambda\downarrow0,
\]
which proves the assertion.
\end{proof}

\subsection{Proof of Theorem~\ref{thm:AsymptoticsGeneralD}}

We start by noting that by Theorem~\ref{thm:CellIntensity} we have
\begin{equation}\label{eq:CellIntensity1}
	\xi_{d-1}(\lambda^{d-1},\lambda) =\frac{\lambda^{(d-1)^2}}{2d}
	\int_0^1 (1-s)^{\frac{d-3}{2}}
	\exp\!\Big(
	-\lambda^{d-1}\frac{\kappa_{d-1}}{d+1}
	s^{\frac{d+1}{2}}
	{}_2F_1\!\Big(\tfrac{d+1}{2},\tfrac{d+1}{2};\tfrac{d+3}{2};s\Big)
	\Big)\,I_d(s)\,\dint s.
\end{equation}
Since the case $d=2$ was treated separately in Corollaries
\ref{cor:CellIntensity_d2} and \ref{cor:lambda_asymptotics_d2}, we may assume $d\ge 3$. As we already mentioned the integral $I_d(s)$ cannot in general be evaluated explicitly. At the same time it is easy to see that $I_d(s)\to\infty$ and ${}_2F_1\!\big(\tfrac{d+1}{2},\tfrac{d+1}{2};\tfrac{d+3}{2};s\big)\to\infty$ as $s\uparrow 1$. Hence, our first step is to determine the corresponding asymptotics and to control the corresponding error terms.

\paragraph{Step 1: Asymptotics for $I_d(s)$ as $s\uparrow 1$, part I.}
Recall that
\[
	I_d(s)
	=\int_{(\RR^{d-1})^d}
	\Delta_{d-1}(\bz_1,\ldots,\bz_d)
	\prod_{i=1}^d(1-\|\bz_i\|^2)^{-\frac{d+1}{2}}
	\,{\bf 1}\{\|\bz_i\|^2\le s\}
	\,\dint\bz_i.
\]
Using polar coordinates $\bz_i=r_i \bu_i$ with $\bu_i\in \SS^{d-2}$ and $r_i\in[0,\sqrt{s}]$, and setting $t=1-s$, we introduce for $i\in\{1,\ldots,d\}$ the change of variables $y_i = t^{-1} (1-r_i^2)$, which implies $\dint r_i=-\frac{t}{2\sqrt{1-ty_i}}\dint y_i$. This yields
\begin{equation*}
	I_d(1-t)
	=t^{-\frac{d(d-1)}{2}}2^{-d}
	\int_{(\SS^{d-2})^d}\int_{[1,1/t]^d}
	G(t,\by,\bu_1,\ldots,\bu_d)\,
	\dint \by\, \sigma(\dint \bu_1)\, \cdots \sigma(\dint \bu_d),
\end{equation*}
where $\dint\by = \dint y_1 \dots \dint y_d$ and the integrand is
\[
	G(t,\by,\bu_1,\ldots,\bu_d):=V(t,\by,\bu_1,\ldots,\bu_d)\prod_{i=1}^d y_i^{-\frac{d+1}{2}}(1-t y_i)^{\frac{d-3}{2}}
\]
with
\[
	V(t,\by,\bu_1,\ldots,\bu_d):=\Delta_{d-1}(\sqrt{1-t y_1}\bu_1,\dots,\sqrt{1-t y_d}\bu_d).
\]

Note that $V(t,\by,\bu_1,\ldots,\bu_d)\leq V_{\rm max}$, where $V_{\rm max}$ is the volume of the regular simplex with vertices on $\SS^{d-2}$. The exact value for the constant $V_{\rm max}$ can be found in \cite[Proposition 4.21]{KSTBook}. Hence, $G(t,\by,\bu_1,\ldots,\bu_d)\leq  V_{\rm max} \prod_{i=1}^d y_i^{-(d+1)/2}$ and by dominated convergence theorem the limit constant is defined by
\begin{align*}
	C_d & := \lim_{t\downarrow 0} t^{\frac{d(d-1)}{2}} I_d(1-t)                                                                                                                   \\
	    & =
	2^{-d} \int_{(\SS^{d-2})^d}\int_{[1,\infty)^d} \Delta_{d-1}(\bu_1,\ldots,\bu_d) \prod_{i=1}^d y_i^{-\frac{d+1}{2}} \,\dint \by\,\sigma(\dint\bu_1)\cdots \,\sigma(\dint\bu_d) \\
	    & =
	2^{-d}\Big(\int_1^\infty y^{-\frac{d+1}{2}}\,\dint y\Big)^d\int_{(\SS^{d-2})^d}\Delta_{d-1}(\bu_1,\ldots,\bu_d)\,\sigma(\dint\bu_1)\cdots\sigma(\dint\bu_d)                   \\
	    & =
	\frac{1}{(d-1)^d}\int_{(\SS^{d-2})^d}\Delta_{d-1}(\bu_1,\ldots,\bu_d)\,\sigma(\dint\bu_1)\cdots \,\sigma(\dint\bu_d).
\end{align*}
Up to the normalizing constant $\kappa_{d-1}^d$, the last expression is the expected volume of a random simplex with vertices uniformly distributed in the $(d-2)$-dimensional unit sphere. It can be evaluated using \cite[Theorem 4.12]{KSTBook} or \cite[Theorem~8.2.3]{SW}, which leads to
\begin{align}
	C_d = \frac{1}{(d-1)^d}\frac{2^{d}\,\pi^{\frac{d(d-1)-1}{2}}}{(d-1)!}\,
	\frac{\Gamma\!\left(\frac{(d-1)^2+1}{2}\right)}{\Gamma\!\left(\frac{(d-1)^2}{2}\right)}\,
	\frac{1}{\Gamma\!\left(\frac d2\right)^{d-1}}.\label{eq:Cd}
\end{align}
To derive the error term, consider the difference
\[
	\delta(t) := 2^d \left(t^{\frac{d(d-1)}{2}} I_d(1-t) - C_d \right).
\]
We split this difference into two parts: the error from the truncated domain and the error from the integrand approximation. Using the triangle inequality we find
\[
	|\delta(t)| \le \delta_1(t) + \delta_2(t)
\]
with
\begin{align*}
	\delta_1(t) & := \int_{(\SS^{d-2})^d}\int_{[1,\infty)^d \setminus [1,1/t]^d} G(0,\by,\bu_1,\ldots,\bu_d)\,\dint \by\,\sigma(\dint\bu_1)\cdots \sigma(\dint\bu_d),          \\
	\delta_2(t) & := \int_{(\SS^{d-2})^d}\int_{[1,1/t]^d} |G(t,\by,\bu_1,\ldots,\bu_d) - G(0,\by,\bu_1,\ldots,\bu_d)|\,\dint \by\,\sigma(\dint \bu_1)\cdots\sigma(\dint\bu_d).
\end{align*}

\paragraph{Step 2: The error term $\delta_1(t)$.} The domain $[1,\infty)^d \setminus [1,1/t]^d$ consists of points where at least one coordinate $y_k$, $k\in\{1,\ldots,d\}$, exceeds the value $1/t$. Since $G(0,\by,\bu_1,\dotsc,\bu_d) \le V_{\rm max} \prod_{i=1}^d y_i^{-(d+1)/2}$ and $\sigma(\SS^{d-2})=(d-1)\kappa_{d-1}$, a union bound over which coordinate $y_k$ exceeds $1/t$ gives
\begin{align*}
	\delta_1(t)
	 & \leq \sum_{k=1}^dV_{\rm max}\,\sigma(\SS^{d-2})^d\Big(\prod_{\substack{i=1        \\i\neq k}}^d\int_1^\infty y_i^{-\frac{d+1}{2}}\,\dint y_i\Big)\int_{1/t}^\infty y_k^{-\frac{d+1}{2}}\,\dint y_k \\
	 & = dV_{\rm max}\,\sigma(\SS^{d-2})^d\,\Big(\frac{2}{d-1}\Big)^d\,t^{\frac{d-1}{2}}
	= 2^dd\,V_{\rm max}\,\kappa_{d-1}^d\,t^{\frac{d-1}{2}}.
\end{align*}
For $d \ge 3$, we have $\frac{d-1}{2} \ge 1$, so $\delta_1(t)$ is $O(t)$.

\paragraph{Step 3: The error term $\delta_2(t)$.} On the domain $[1, 1/t]^d$, we have $t y_i \in [0,1]$.
Using for $0\leq x\leq 1$ the elementary inequality
\begin{equation}\label{eq:ElementaryInequality2}
	|(1-x)^\alpha - 1| \le \max\{1,\alpha\} x\qquad\text{for $\alpha > 0$}
\end{equation}
we can bound the difference in the integrand in $\delta_2(t)$.
The volume $\Delta_{d-1}(\bx_1,\ldots,\bx_d)$ is the absolute value of a polynomial function of the coordinates of $\bx_1,\ldots,\bx_d\in\BB^{d-1}$ and hence Lipschitz as a function of $(\bx_1,\ldots,\bx_d)$ for some Lipschitz constant $L=L(d)$. That is,
$$
	|\Delta_{d-1}(\bx_1,\ldots,\bx_d)-\Delta_{d-1}(\bx_1',\ldots,\bx_d')| \leq L\sum_{i=1}^d\|\bx_i-\bx_i'\|,\qquad \bx_1,\ldots,\bx_d,\bx_1',\ldots,\bx_d'\in\BB^{d-1}.
$$
In fact, one can take $L=\frac{2^{d-2}}{(d-2)!}$.
Applying this with $\bx_i=\sqrt{1-ty_i}\bu_i$ and $\bx_i'=\bu_i$, $i\in\{1,\ldots,d\}$, we get
\begin{align*}
	|V(t,\by,\bu_1,\dotsc,\bu_d) - V(0,\by,\bu_1,\dotsc,\bu_d)|
	\le
	L \sum_{i=1}^d |\sqrt{1-ty_i} - 1|
	\le
	L t \sum_{i=1}^d y_i,
\end{align*}
where we used \eqref{eq:ElementaryInequality2} with $\alpha=1/2$ in the last step.
Similarly, we can deal with the product term $P(t,\by) = \prod_{i=1}^d (1-ty_i)^{(d-3)/2}$, where we distinguish the cases $d=3$ and $d>3$.
\begin{itemize}
	\item If $d=3$, the exponent is 0, so $P(t,\by)=1$ and there is no error from this term.
	\item If $d > 3$, then
	      \begin{align*}
		      0\leq 1-P(t,\by)
		       & = 1-\prod_{i=1}^d (1-ty_i)^{(d-3)/2}
		      \leq \sum_{i=1}^d \Big(1-(1-ty_i)^{(d-3)/2}\Big)               \\
		       & \leq \max\left\{1,\frac{d-3}{2}\right\} t \sum_{i=1}^d y_i,
	      \end{align*}
	      where we used \eqref{eq:ElementaryInequality2} and the elementary inequality
	      \[
		      1-\prod_{i=1}^d a_i = (1-a_1) + \sum_{i=2}^d \left(\prod_{j=1}^{i-1} a_j\right) (1-a_i) \leq \sum_{i=1}^d (1-a_i)
	      \]
	      for $0\leq a_i\leq 1$ for all $i\in\{1,\ldots,d\}$.
\end{itemize}
Hence for every $d\geq 3$ we have
\[
	|P(t,\by) - 1| \le \max\left\{1,\frac{d-3}{2}\right\} t \sum_{i=1}^d y_i.
\]
Thus, the integrand in $\delta_2(t)$ is bounded by
\begin{align*}
	 & |G(t,\by,\bu_1,\ldots,\bu_d) - G(0,\by,\bu_1,\ldots,\bu_d)|                                                                                  \\
	 & \qquad \leq \left(\prod_{i=1}^d y_i^{-\frac{d+1}{2}}\right) \left(|V(t,\by,\bu_1,\ldots,\bu_d)P(t,\by) - V(0,\by,\bu_1,\dotsc,\bu_d)|\right) \\
	 & \qquad \leq \left(\prod_{i=1}^d y_i^{-\frac{d+1}{2}}\right) \left(|V(t,\by,\bu_1,\ldots,\bu_d)| |P(t,\by) - 1|
	+ |V(t,\by,\bu_1,\dotsc,\bu_d)-V(0,\by,\bu_1,\dotsc,\bu_d)|\right)                                                                              \\
	 & \qquad \le (V_{\max}\max\{1,\tfrac{d-3}{2}\}+ L) \, t \left(\sum_{k=1}^d y_k\right) \prod_{i=1}^d y_i^{-\frac{d+1}{2}}.
\end{align*}
Integrating this over $[1, 1/t]^d$ yields, for some constant $0<c<\infty$ only depending on $d$ the bound
\begin{align*}
	\delta_2(t)\leq c\,t\int_{[1,1/t]^d}\left(\sum_{k=1}^d y_k\right) \prod_{i=1}^d y_i^{-\frac{d+1}{2}}\,\dint\by=d\,c\,t\Big(\int_1^{1/t}y^{-\frac{d-1}{2}}\,\dint y\Big)\Big(\int_1^{1/t}y^{-\frac{d+1}{2}}\,\dint y\Big)^{d-1}.
\end{align*}
Now,
$$
	\int_1^{1/t}y^{-\frac{d+1}{2}}\,\dint y=\frac{2}{d-1}(1-t^{\frac{d-1}{2}})\qquad\text{and}\qquad\int_1^{1/t}y^{-\frac{d-1}{2}}\,\dint y = \begin{cases}
		\log\frac{1}{t},                    & d=3,     \\
		\frac{2}{d-3}(1-t^{\frac{d-3}{2}}), & d\geq 4,
	\end{cases}
$$
and it follows that
$$
	\delta_2(t) \leq \begin{cases}
		c'\,t\,\log\frac{1}{t}, & d=3,    \\
		c'\,t,                  & d\geq 4
	\end{cases}
$$
with another constant $0<c'<\infty$ only depending on $d$.

\paragraph{Step 4: Asymptotics for $I_d(s)$ as $s\uparrow 1$, part II.}
The definition of the error term $\delta(t)$ gives
\begin{equation}\label{eq:Id_asymptotic}
	I_d(1-t)= C_d\,t^{-\frac{d(d-1)}{2}}\Big(1+R_I(t)\Big),
\end{equation}
where $R_I(t) := \frac{\delta(t)}{C_d2^d}$.
Combining the error estimates for $\delta_1(t)$ and $\delta_2(t)$ developed in the two previous steps, we obtain the asymptotic
$R_I(t) = O(t)$ for $d > 3$ and $R_I(t) = O(t \log(1/t))$ for $d=3$ as $t\downarrow 0$. In all dimensions $d \ge 3$, $R_I(t) \to 0$ as $t \to 0$.

\paragraph{Step 5: Asymptotics for the hypergeometric function.}

Regarding the hypergeometric function in Theorem~\ref{thm:AsymptoticsGeneralD}, we use the Euler transformation
\[
	{}_2F_1(a,b;c;s)
	=
	(1-s)^{c-a-b}
		{}_2F_1(c-a,c-b;c;s),
\]
see \cite[Equation~15.8.1]{NIST}. With $a=b=\frac{d+1}{2}$ and $c=\frac{d+3}{2}$ this gives
\[
	{}_2F_1\!\Big(\tfrac{d+1}{2},\tfrac{d+1}{2};
	\tfrac{d+3}{2};s\Big)
	=
	(1-s)^{-\frac{d-1}{2}}
		{}_2F_1\!\Big(1,1;\tfrac{d+3}{2};s\Big).
\]
By Gauss' summation formula \cite[Equation~15.4.20]{NIST} at $s=1$, ${}_2F_1(1,1;\tfrac{d+3}{2};1)=\frac{d+1}{d-1}$.
Thus, for all $d\ge3$,
\[
	{}_2F_1\!\Big(\tfrac{d+1}{2},\tfrac{d+1}{2};
	\tfrac{d+3}{2};s\Big)
	=
	\frac{d+1}{d-1}\,
	(1-s)^{-\frac{d-1}{2}}(1+o(1)),
	\qquad s\uparrow1.
\]

For the error term one has to distinguish $d=3$ from $d\ge4$. If
$d=3$, then
\[
	{}_2F_1(2,2;3;s)
	=
	\frac{2}{s^2}
	\left(
	\frac{1}{1-s}+\log(1-s)-1
	\right),
\]
which follows from
${}_2F_1(1,1;2;s)=-\frac{1}{s}\log(1-s)$ and the derivative formula
\cite[Equations~15.4.1 and~15.5.1]{NIST}. Hence
\begin{equation}\label{eq:HypergeomAsympt1}
	{}_2F_1(2,2;3;s)
	=
	2(1-s)^{-1}
	\Big(
	1
	+
	(1-s)\big(\log(1-s)+1\big)
	+
	O\big((1-s)^2|\log(1-s)|\big)
	\Big),
	\qquad s\uparrow1.
\end{equation}

If $d\ge4$, then
${}_2F_1(1,1;\tfrac{d+3}{2};s)$ is differentiable at $s=1$, again by
Gauss' summation formula applied after differentiation. Therefore
\[
	{}_2F_1\!\Big(1,1;\tfrac{d+3}{2};s\Big)
	=
	\frac{d+1}{d-1}
	-
	\frac{2(d+1)}{(d-1)(d-3)}(1-s)
	+o(1-s),
	\qquad s\uparrow1.
\]
Combining this expansion with the Euler transformation above and factoring
out \((d+1)/(d-1)\), we obtain, for \(d\ge4\),
\begin{equation}\label{eq:HypergeomAsympt2}
	{}_2F_1\!\Big(\tfrac{d+1}{2},\tfrac{d+1}{2};
	\tfrac{d+3}{2};s\Big)
	=
	\frac{d+1}{d-1}
	(1-s)^{-\frac{d-1}{2}}
	\left(
	1-\frac{2}{d-3}(1-s)+o(1-s)
	\right),
	\qquad s\uparrow1.
\end{equation}

\paragraph{Step 6: Global change of variables and limit transition.}
We substitute $s = 1 - \lambda^2 x$ in \eqref{eq:CellIntensity1}. Then the integration domain transforms from $s \in (0,1)$ to $x \in (0, \lambda^{-2})$ and the cell intensity becomes
\[
	\xi_{d-1}(\lambda^{d-1}, \lambda) = \frac{\lambda^{(d-1)^2}}{2d} \int_0^{\lambda^{-2}} (\lambda^2 x)^{\frac{d-3}{2}} \exp\!\Big( - \lambda^{d-1} A(1-\lambda^2 x) \Big) \, I_d(1-\lambda^2 x) \, \lambda^2 \dint x,
\]
where
\begin{equation}\label{eq:A}
	A(s) := \frac{\kappa_{d-1}}{d+1} s^{\frac{d+1}{2}}
		{}_2F_1\!\Big(\tfrac{d+1}{2},\tfrac{d+1}{2};\tfrac{d+3}{2};s\Big).
\end{equation}
We now insert the explicit asymptotic forms derived in Steps 4 and 5. Recall from \eqref{eq:Id_asymptotic}, \eqref{eq:HypergeomAsympt1} and \eqref{eq:HypergeomAsympt2} that for $t \in (0,1)$,
\begin{align}
	I_d(1-t) & = C_d \, t^{-\frac{d(d-1)}{2}} \big(1 + R_I(t)\big),  \notag                                                                   \\
	A(1-t)   & = B_d \, t^{-\frac{d-1}{2}} \big(1 + R_A(t)\big), \qquad \text{with}\qquad B_d = \frac{\kappa_{d-1}}{d-1},\label{eq:Aestimate}
\end{align}
where the error terms satisfy $R_I(t) \to 0$ and $R_A(t) \to 0$ as $t \downarrow 0$. Moreover, the explicit form of the error bounds we derived ensures that $R_I$ and $R_A$ are bounded on the interval $(0,1)$.

After the substitution $t=\lambda^2 x$, the powers of $\lambda$ cancel.
Indeed, the prefactor contributes the factor $\lambda^{(d-1)^2}$. The term
$(1-s)^{(d-3)/2}$, with $1-s=\lambda^2 x$, contributes
$\lambda^{d-3}$. The differential contributes another factor
$\lambda^2$, while the asymptotic
\[
	I_d(1-\lambda^2 x)
	=
	C_d(\lambda^2 x)^{-\frac{d(d-1)}2}(1+o(1))
\]
contributes $\lambda^{-d(d-1)}$. Hence the total power of $\lambda$ is $(d-1)^2+(d-3)+2-d(d-1)=0$.
The cell intensity can thus be written explicitly in terms of the error functions $R_I$ and $R_A$ as
\begin{equation}\label{eq:CellIntensityIntegral}
	\xi_{d-1}(\lambda^{d-1}, \lambda) = \int_0^{\infty} f_\lambda(x) \,\dint x
\end{equation}
with
\[
	f_\lambda(x) := \frac{C_d}{2d} \, x^{-1 - \frac{(d-1)^2}{2}} \big(1 + R_I(\lambda^2 x)\big) \, \exp\!\Big( - B_d x^{-\frac{d-1}{2}} \big(1 + R_A(\lambda^2 x)\big) \Big) \, {\bf 1}_{(0, \lambda^{-2})}(x).
\]

For $\lambda\downarrow0$ the integrand converges pointwise for all $x\in(0,\infty)$, that is,
\[
	\lim_{\lambda \downarrow 0} f_\lambda(x) = \frac{C_d}{2d} \, x^{-1 - \frac{(d-1)^2}{2}} \exp\!\Big( - B_d x^{-\frac{d-1}{2}} \Big).
\]
Next, we take the limit on both sides of \eqref{eq:CellIntensityIntegral} as $\lambda\downarrow0$. To justify the exchange of limit and integration on the right-hand side we need to construct an integrable majorant using the properties of $R_I$ and $R_A$.

\begin{itemize}
	\item \textit{The polynomial factor:}
	      Since $|R_I(t)|=\frac{|\delta(t)|}{2^d C_d}\leq \frac{\delta_1(t)+\delta_2(t)}{2^d C_d}$ we can use the bounds established in Step 2 and Step 3 to conclude that
	      \[
		      |1 + R_I(\lambda^2 x)| \le c'' \quad \text{for all } \lambda^2 x \in (0,1),
	      \]
	      for some constant $0<c''<\infty$ only depending on $d$.

	\item \textit{The exponential factor:} Since $R_A(t) \to 0$ for $t\downarrow 0$, we can choose $t_0>0$ small enough such that for all $t\in (0,t_0)$ we have $|R_A(t)| \le 1/2$. Therefore,
	      \[
		      1 + R_A(t) \ge \frac{1}{2} \qquad \text{for all $t\in(0,t_0)$.}
	      \]
	      This implies that the exponential term is bounded by
	      \[
		      \exp\!\Big( - B_d x^{-\frac{d-1}{2}} (1 + R_A(\lambda^2 x)) \Big) \le \exp\!\Big( - \frac{B_d}{2} x^{-\frac{d-1}{2}} \Big),
	      \]
	      for all $\lambda^2 x \in (0,t_0)$. Furthermore, for $x\geq \frac{t_0}{\lambda^2}$, we can use the bound
	      \begin{equation*}
		      \exp\!\Big( - B_d x^{-\frac{d-1}{2}} (1 + R_A(\lambda^2 x)) \Big) \leq c''',
	      \end{equation*}
	      for some constant $0<c'''<\infty$ only depending on $d$.
\end{itemize}

Combining these bounds, the integrand $f_\lambda$ is dominated in absolute value by
\begin{align*}
	|f_\lambda(x)| & \le \begin{cases}
		                     \frac{C_d \,c''}{2d} \, x^{-1 - \frac{(d-1)^2}{2}} \exp\!\Big( - \frac{B_d}{2} x^{-\frac{d-1}{2}} \Big), & x \in (0,t_0/\lambda^2),   \\
		                     \frac{C_d c'''}{2d} x^{-1-\frac{(d-1)^2}{2}},                                                            & x\in[t_0/\lambda^2,\infty)
	                     \end{cases} \\[0.25cm]
	               & \leq \frac{C_d\max\{c'',c'''\}}{2d} x^{-1-\frac{(d-1)^2}{2}}\Big(\exp\!\Big( - \frac{B_d}{2} x^{-\frac{d-1}{2}} \Big) + {\bf 1}\{x\geq 1\}\Big),
\end{align*}
provided $0<\lambda^2<t_0$, which we can assume without loss of generality. This function is integrable on $(0, \infty)$ because the exponent of $x$ is $<-1$ (ensuring convergence at infinity) and the exponential term ensures rapid decay at $x=0$. Thus, by the dominated convergence theorem,
\[
	\lim_{\lambda \to 0} \xi_{d-1}(\lambda^{d-1}, \lambda) = \int_0^\infty \lim_{\lambda \to 0} f_\lambda(x) \, \dint x = \frac{C_d}{2d} \int_0^\infty x^{-1 - \frac{(d-1)^2}{2}} \exp\!\Big( - B_d x^{-\frac{d-1}{2}} \Big) \, \dint x.
\]

\paragraph{Step 7: Evaluation of the integral and final simplification.}
It remains to compute the integral
\[
	J := \int_0^\infty x^{-1 - \frac{(d-1)^2}{2}} \exp\!\Big( - B_d x^{-\frac{d-1}{2}} \Big) \, \dint x,
\]
where we recall that $B_d = \frac{\kappa_{d-1}}{d-1}$. We use the substitution $u = x^{-\frac{d-1}{2}}$, which yields
\begin{align*}
	J & = \int_0^\infty u^{\frac{2}{d-1} + (d-1)} e^{-B_du} \cdot \frac{2}{d-1} u^{-\frac{2}{d-1}-1} \, \dint u
	= \frac{2}{d-1} \int_0^\infty u^{(d-1)-1} e^{-B_du} \, \dint u                                              \\
	  & = \frac{2}{d-1} B_d^{-(d-1)} \Gamma(d-1) = 2 \kappa_{d-1}^{-(d-1)} (d-1)^{d-2} \Gamma(d-1).
\end{align*}
Finally, we combine this with the prefactor $\frac{C_d}{2d}$ from Step 6 to obtain the limit
\[
	\lim_{\lambda\downarrow 0} \xi_{d-1}(\lambda^{d-1},\lambda)
	= \frac{C_d}{2d} \cdot J
	= \frac{C_d}{d} \Gamma(d-1) \frac{(d-1)^{d-2}}{\kappa_{d-1}^{d-1}}.
\]
We now insert the explicit value of $C_d$ in \eqref{eq:Cd}, express $\kappa_{d-1}$ in terms of a Gamma function and simplify. This gives the final result
\[
	\lim_{\lambda\downarrow 0} \xi_{d-1}(\lambda^{d-1},\lambda)
	= \frac{2^d}{d(d-1)^3} \pi^{\frac{d-2}{2}} \frac{\Gamma(\frac{(d-1)^2+1}{2})}{\Gamma(\frac{(d-1)^2}{2})} \left( \frac{\Gamma(\frac{d+1}{2})}{\Gamma(\frac{d}{2})} \right)^{d-1},
\]
and completes the proof.\qed

\subsection{Proofs of Theorem~\ref{thm:SkeletonConvergence}, Theorem~\ref{thm:LocalVolume} and its corollaries}

\begin{proof}[Proof of Theorem~\ref{thm:SkeletonConvergence}]
	The result is a consequence of \cite[Theorem~3.6 and
		Theorem~3.8]{GiWL26}. More precisely, let for each $n\in\NN$, $\eta_n$ be a Poisson point
	process on $\RR^{d-1}\times[0,\infty)$ with intensity measure of the
	form
	\[
		\Lambda_n(\cdot)
		=
		\int_{\RR^{d-1}}\int_0^\infty
		f_n(h)\,{\bf 1}\{(\bv,h)\in\cdot\}\,\dint h\,\dint\bv,
	\]
	where $(f_n)_{n\in\NN}$ is a sequence of locally integrable functions $f_n:[0,\infty)\to[0,\infty)$
	satisfying
	\begin{equation}\label{eq:Condition}
		\lim_{n\to\infty}\int_0^x f_n(h)\,\dint h
		=
		c\in(0,\infty)
		\qquad\text{for all }x>0 .
	\end{equation}
	If $\cL^*(\eta_n)$ denotes the dual Poisson--Laguerre tessellation
	induced by $\eta_n$, then the distribution of the skeleton of
	$\cL^*(\eta_n)$ converges weakly to the distribution of the skeleton of
	the classical Poisson--Delaunay tessellation of intensity $c$, by
	\cite[Theorem~3.8]{GiWL26}. Moreover, the convergence holds in the
	stronger local sense by \cite[Theorem~3.6]{GiWL26}.

	Hence, by Theorem~\ref{thm:DlambdaIsTessellation} and
	\eqref{eq:IntensityEta}, it remains to check \eqref{eq:Condition} for
	the height densities arising from our model. Let
	$(\lambda_n)_{n\in\NN}$ be a sequence of positive real numbers with
	$\lambda_n\downarrow0$. We shift the transformed process vertically by
	$-\lambda_n^2$. This does not change the induced dual Laguerre
	tessellation, since a common vertical shift of all points only shifts the
	corresponding paraboloids vertically and leaves their spatial apexes and
	empty-paraboloid relations unchanged. After this shift the height density
	is
	\begin{equation}\label{eq:fnh}
		f_n(h)
		=
		\frac{1}{2}\lambda_n^{d-1}
		(h+\lambda_n^2)^{-\frac{d+1}{2}}
		{\bf 1}\{h\ge0\}.
	\end{equation}
	For every $x>0$, we have, as $n\to\infty$,
	\begin{align*}
		\int_0^x f_n(h)\,\dint h
		 & =
		\frac{1}{2}\lambda_n^{d-1}
		\int_0^x
		(h+\lambda_n^2)^{-\frac{d+1}{2}}\,\dint h  =
		\frac{1}{d-1}
		\Big[
		1-
		\Big(\frac{\lambda_n^2}{x+\lambda_n^2}\Big)^{\frac{d-1}{2}}
		\Big]
		\longrightarrow
		\frac{1}{d-1}.
	\end{align*}
	Thus \eqref{eq:Condition} holds with $c=\frac{1}{d-1}$, and the cited
	convergence theorem for Poisson--Laguerre tessellations yields the
	claim.
\end{proof}

\begin{proof}[Proof of Corollary~\ref{cor:AllFaceIntensities}] 
Consider a stationary, locally finite, face-to-face simplicial tessellation $\cT$ of $\RR^{d-1}$ with positive and finite cell intensity. For $k\in\{0,\ldots,d-1\}$ we denote by $\xi_k(\cT)$ the intensity of $k$-faces and let $Z_{\cT}$ denote the typical cell of $\cT$. From \cite[Theorem~10.1.3]{SW} it follows that
\begin{equation}\label{eq:FaceIntensityAngleIdentity}
    \xi_k(\cT)
    =\xi_{d-1}(\cT)\,
    \EE\sigma_{k+1}(Z_{\cT}).
\end{equation}
Moreover, a
$(d-1)$-dimensional simplex has exactly $\binom{d}{k+1}$
faces of dimension $k$, and consequently
\begin{equation}\label{eq:SimplexAngleSumBound}
    0\leq\sigma_{k+1}(P)\leq\binom{d}{k+1}.
\end{equation}
Together with \eqref{eq:FaceIntensityAngleIdentity}, this implies that $\xi_k(\cT)$ is finite for all $k\in\{0,\ldots,d-1\}$.

We also note that for each $k\in\{0,\ldots,d-1\}$ the functional $\sigma_{k+1}$ is continuous on the space
of nondegenerate $(d-1)$-dimensional simplices equipped with
the Hausdorff metric. Indeed, the tangent cones at corresponding faces in the definition of $\sigma_{k+1}$ converge
in the conic Hausdorff metric and their solid angles depend
continuously on the cones, see \cite[Proposition~8.2(2)]{MT14}.

We now apply relation \eqref{eq:FaceIntensityAngleIdentity} to
$\cD_{\lambda^{d-1},\lambda}$ and to
$\cD_{1/(d-1)}$. By Corollary~\ref{cor:TypicalCellConvergence}, $Z_\lambda$ converges in distribution to $Z^{\operatorname{PD}}_{1/(d-1)}$ as $\lambda\downarrow0$, and the limiting cell is almost surely a nondegenerate
$(d-1)$-dimensional simplex. The continuous mapping theorem
therefore implies
\[
    \sigma_{k+1}(Z_\lambda)
    \ \xrightarrow{\mathrm d}\
    \sigma_{k+1}\bigl(Z^{\operatorname{PD}}_{1/(d-1)}\bigr),\qquad\lambda\downarrow0.
\]
By \eqref{eq:SimplexAngleSumBound}, these random variables
all take values in the same compact interval. Weak
convergence consequently implies that $\EE\sigma_{k+1}(Z_\lambda)$ converges to $\EE\sigma_{k+1}\bigl(Z^{\operatorname{PD}}_{1/(d-1)}\bigr)$ as $\lambda\downarrow0$.
Combining this with
Theorem~\ref{thm:AsymptoticsGeneralD} and
\eqref{eq:FaceIntensityAngleIdentity}, we obtain
\[
\begin{aligned}
    &\xi_k(\lambda^{d-1},\lambda)
    =
    \xi_{d-1}(\lambda^{d-1},\lambda)\,
    \EE\sigma_{k+1}(Z_\lambda)\\
    &\qquad\longrightarrow
    \xi_{d-1}^{\operatorname{PD}}
    \Big(\frac{1}{d-1}\Big)\,
    \EE\sigma_{k+1}
    \bigl(Z^{\operatorname{PD}}_{1/(d-1)}\bigr)
    =\xi_k^{\operatorname{PD}}
    \Big(\frac{1}{d-1}\Big),\qquad\lambda\downarrow0.
\end{aligned}
\]
This proves the assertion.
\end{proof}

\begin{proof}[Proof of Theorem~\ref{thm:LocalVolume}]
	Fix $\ell\in\{0,1\}$. We use the same parametrization as in the proof
	of Theorem~\ref{thm:CellIntensity}. Let
	$(\bv_1,h_1),\ldots,(\bv_d,h_d)$ generate a cell and let $\Pi_{(\bw,t^2)}=\{(\bv,h):h\le t^2-\|\bv-\bw\|^2\}$ be the corresponding empty paraboloid. In the upper halfspace model, the
	associated facet is contained in the hemisphere $y(\bv)=\sqrt{t^2-\|\bv-\bw\|^2}$.

	The contribution of the projected simplex
	$[\bv_1,\ldots,\bv_d]$ to the numerator of
	$F_{\gamma,\lambda}^{(\ell)}(W)$ can be written simultaneously as
	\[
		(d-1)^{\ell-1}
		\int_{[\bv_1,\ldots,\bv_d]\cap W}
		t^\ell
		\bigl(t^2-\|\bv-\bw\|^2\bigr)^{-\frac{d-1+\ell}{2}}
		\,\dint\bv.
	\]
	For $\ell=0$, this follows by integrating the hyperbolic volume element
	$y^{-d}\,\dint\bv\,\dint y$ above the hemisphere. For $\ell=1$, we use the Riemannian area element induced on the
	supporting hemisphere. More precisely, consider the graph parametrization $\varphi(\bv):=(\bv,y(\bv))$.
	Since the hyperbolic metric in the upper halfspace model is given by \eqref{eq:HyperbolicMetric}
	the metric induced on the graph is represented by the matrix
	\[
		y(\bv)^{-2}
		\left(
		I_{d-1}+\nabla y(\bv)\nabla y(\bv)^{\mathsf T}
		\right),
	\]
	where $I_{d-1}$ is the $(d-1)\times(d-1)$ identity matrix.
	Using $\det\bigl(I_{d-1}+aa^{\mathsf T}\bigr)=1+\|a\|^2$
	the corresponding hyperbolic $(d-1)$-dimensional area element is
	\[
		y(\bv)^{-(d-1)}
		\sqrt{1+\|\nabla y(\bv)\|^2}\,\dint\bv.
	\]
	For the hemisphere under consideration we have $\nabla y(\bv)=-\frac{\bv-\bw}{y(\bv)}$, and hence
	\[
		1+\|\nabla y(\bv)\|^2
		=
		1+\frac{\|\bv-\bw\|^2}{y(\bv)^2}
		=
		\frac{t^2}{y(\bv)^2},
	\]
	where we used
	$y(\bv)^2+\|\bv-\bw\|^2=t^2$. Consequently,
	\[
		y(\bv)^{-(d-1)}
		\sqrt{1+\|\nabla y(\bv)\|^2}\,\dint\bv
		=
		t\,y(\bv)^{-d}\,\dint\bv
		=
		t\bigl(t^2-\|\bv-\bw\|^2\bigr)^{-d/2}\,\dint\bv.
	\]
	Therefore, the contribution of the facet above
	$[\bv_1,\ldots,\bv_d]\cap W$ to the local surface area is
	\[
		t\int_{[\bv_1,\ldots,\bv_d]\cap W}
		\bigl(t^2-\|\bv-\bw\|^2\bigr)^{-d/2}\,\dint\bv.
	\]

	We now apply the multivariate Mecke equation and the transformation $G$
	from the proof of Theorem~\ref{thm:CellIntensity}. After putting
	$\bv=\bw+t\bu$, the powers of $t$ in the preceding contribution
	cancel and it becomes
	\[
		(d-1)^{\ell-1}
		\int_{[\bz_1,\ldots,\bz_d]}
		{\bf 1}\{\bw+t\bu\in W\}
		(1-\|\bu\|^2)^{-\frac{d-1+\ell}{2}}
		\,\dint\bu.
	\]
	Integration with respect to $\bw$ gives the factor
	$\vol_{\RR^{d-1}}(W)$. Dividing by
	$\vol_{\HH^d}(C_\lambda(W))$ if $\ell=0$ and by
	$\cH_{\HH^d}^{d-1}(W_\lambda)$ if $\ell=1$, and repeating the steps of the proof of Theorem \ref{thm:CellIntensity}, we obtain
	\begin{align*}
		\EE F_{\gamma,\lambda}^{(\ell)}(W)
		 & =
		\frac{\gamma^d\lambda^{d-1}}{d}
		\int_\lambda^\infty
		t^{-d}
		\exp\bigl(
		-\Lambda_{\gamma,\lambda}(\Pi_{(o,t^2)})
		\bigr)         \\
		 & \quad\times
		\int_{(\RR^{d-1})^d}
		\Delta_{d-1}(\bz_1,\ldots,\bz_d)\left(
		\int_{[\bz_1,\ldots,\bz_d]}
			(1-\|\bu\|^2)^{-\frac{d-1+\ell}{2}}
		\,\dint\bu
		\right)
		\\
		 & \quad\times
		\prod_{i=1}^d
		(1-\|\bz_i\|^2)^{-\frac{d+1}{2}}
			{\bf 1}\Bigl\{
		1-\|\bz_i\|^2\ge\frac{\lambda^2}{t^2}
		\Bigr\}\,\dint\bz_i\,\dint t .
	\end{align*}

	The void probability is the same as in the proof of
	Theorem~\ref{thm:CellIntensity}. Finally, with $s=1-\frac{\lambda^2}{t^2}$, $t=\frac{\lambda}{\sqrt{1-s}}$
	we have
	\[
		t^{-d}\,\dint t
		=
		\frac{\lambda^{1-d}}{2}
		(1-s)^{\frac{d-3}{2}}\,\dint s
		\qquad\text{and}\qquad
		{\bf 1}\Bigl\{
		1-\|\bz_i\|^2\ge\frac{\lambda^2}{t^2}
		\Bigr\}
		=
		{\bf 1}\{\|\bz_i\|^2\le s\}.
	\]
	Substitution gives the asserted formula with $J_{d,\ell}(s)$.
\end{proof}

\begin{proof}[Proof of Corollary~\ref{cor:Efron}]
	Since the transformation
	$T(\bv,y)=(\bv,y^2)$ leaves the spatial coordinate unchanged, the vertices
	of $\cD_{\gamma,\lambda}$ may equivalently be counted through the points
	of $\eta_{\gamma,\lambda}$. By the general-position property of the
	Poisson point process, a point $(\bv,y)\in\eta_{\gamma,\lambda}$ gives rise to a
	vertex of $\cD_{\gamma,\lambda}$ if and only if it is an extreme point of
	the geodesic convex hull $K_{\gamma,\lambda}$. Hence, by \eqref{eqn:k-intensity},
	\[
		\xi_0(\gamma,\lambda)
		=
		\frac{1}{\vol_{\RR^{d-1}}(W)}\,
		\EE\sum_{(\bv,y)\in\eta_{\gamma,\lambda}}
		{\bf 1}\{\bv\in W\}
		{\bf 1}\big\{
		(\bv,y)\notin
		\conv\big(\eta_{\gamma,\lambda}\setminus\{(\bv,y)\}\big)
		\big\}.
	\]
	Applying the Mecke equation to the Poisson point process
	$\eta_{\gamma,\lambda}$ gives
	\[
		\begin{aligned}
			\xi_0(\gamma,\lambda)
			 & =
			\frac{\gamma}{\vol_{\RR^{d-1}}(W)}\,
			\int_{C_\lambda(W)}
			\PP((\bv,y)\notin K_{\gamma,\lambda})
			\,\vol_{\HH^d}(\dint(\bv,y)).
		\end{aligned}
	\]
	Indeed, after removing the distinguished point in the sum, the remaining
	process has the same distribution as $\eta_{\gamma,\lambda}$, and an
	inserted point $\bx$ is an extreme point of
	$\conv(\eta_{\gamma,\lambda}\cup\{\bx\})$ precisely when
	$\bx\notin K_{\gamma,\lambda}$. Finally, Tonelli's theorem yields
	\[
		\begin{aligned}
			\xi_0(\gamma,\lambda)
			 & =
			\frac{\gamma}{\vol_{\RR^{d-1}}(W)}\,
			\EE\int_{C_\lambda(W)}
			{\bf 1}\big\{
			(\bv,y)\notin K_{\gamma,\lambda}
			\big\}
			\,\vol_{\HH^d}(\dint(\bv,y)) \\
			 & =
			\frac{\gamma}{\vol_{\RR^{d-1}}(W)}\,
			\EE\vol_{\HH^d}\big(
			(B_\lambda^\infty\setminus K_{\gamma,\lambda})
			\cap C_\lambda(W)
			\big),
		\end{aligned}
	\]
	which proves the claim by \eqref{eqn:VolumeCylinder}.
\end{proof}

\begin{proof}[Proof of Corollary~\ref{cor:LocalVolumeAsymptotics}]
	We only indicate the changes compared with the proof of
	Theorem~\ref{thm:AsymptoticsGeneralD}.
	We first consider the case \(d=2\) and \(\ell=0\). Since
	\(\cD_{\gamma,\gamma}\) is a stationary tessellation of \(\RR\) into
	intervals, its vertex and cell intensities coincide. Hence, by
	Corollary~\ref{cor:lambda_asymptotics_d2},
	\[
		\xi_0(\gamma,\gamma)
		=
		\xi_1(\gamma,\gamma)
		=
		1-\pi\gamma+o(\gamma),
		\qquad \gamma\downarrow0.
	\]
	On the other hand, Corollary~\ref{cor:Efron} and
	\(\vol_{\HH^2}(C_\gamma([0,1]))=1/\gamma\) give
	\[
		\xi_0(\gamma,\gamma)
		=
		1-\gamma\,
		\EE\vol_{\HH^2}\bigl(
		K_{\gamma,\gamma}\cap C_\gamma([0,1])
		\bigr).
	\]
	Consequently, $\EE\vol_{\HH^2}\bigl(
		K_{\gamma,\gamma}\cap C_\gamma([0,1])
		\bigr)=\pi+o(1)$,
	and therefore $\EE V_{\gamma,\gamma}([0,1])=\pi\gamma+o(\gamma)$.
	Since \(\EE V_{\gamma,\lambda}(W)\) is independent of \(W\) and
	\(\lambda\) by Theorem~\ref{thm:LocalVolume}, and since
	\(C_{2,0}=4\pi\) by \eqref{eq:C20}, this is precisely
	\[
		\EE V_{\gamma,\lambda}(W)
		=
		C_{2,0}\,
		\frac{\gamma}{4}
		+
		o(\gamma),
	\]
	as asserted. We may therefore assume \(d\ge3\) in what follows, and fix \(\ell\in\{0,1\}\).

	By Theorem~\ref{thm:LocalVolume} and
	\eqref{eqn:VolumeCylinder}, we have
	\begin{align}
		\frac{1}{\gamma}\EE F_{\gamma,\lambda}^{(\ell)}(W)
		 & =
		\frac{\gamma^{d-1}}{2d}
		\int_0^1
		(1-s)^{\frac{d-3}{2}}
		\exp\bigl(-\gamma A(s)\bigr)
		J_{d,\ell}(s)\,\dint s,
		\label{eq:LocalFunctionalAsymptotics}
	\end{align}
	where $A(s)$ is defined by \eqref{eq:A} and by  \eqref{eq:Aestimate} we have
	\[
		A(1-t)
		=
		\frac{\kappa_{d-1}}{d-1}\,
		t^{-\frac{d-1}{2}}\big(1 + R_A(t)\big),
		\qquad t\downarrow0,
	\]
	with $R_A(t) \to 0$ as $t \downarrow 0$ and $R_A$ bounded on $(0,1)$.

	It remains to determine the asymptotic behaviour of
	$J_{d,\ell}(s)$. We use the same polar substitution as in the proof of
	Theorem~\ref{thm:AsymptoticsGeneralD}, namely
	\[
		\bz_i=\sqrt{1-t y_i}\,\bu_i,
		\qquad
		\bu_i\in\SS^{d-2},
		\qquad
		y_i\in[1,1/t].
	\]
	The only additional factor, compared with $I_d(1-t)$, is
	\begin{equation}\label{eq:Auxiliary}
		\int_{
			[\sqrt{1-t y_1}\bu_1,\ldots,
					\sqrt{1-t y_d}\bu_d]}
		(1-\|\bu\|^2)^{-\frac{d-1+\ell}{2}}
		\,\dint\bu.
	\end{equation}
	This factor admits the following geometric interpretation. For
	affinely independent $\bz_1,\ldots,\bz_d\in\inter\BB^{d-1}$ put
	\[
		\bp_i
		:=
		\bigl(\bz_i,\sqrt{1-\|\bz_i\|^2}\bigr)\in\HH^d,
		\qquad i\in\{1,\ldots,d\}.
	\]
	These points lie on the Euclidean unit hemisphere
	$\Sigma:=\{(\bu,y)\in\HH^d:\|\bu\|^2+y^2=1\}$, which represents a
	totally geodesic hypersurface in the upper halfspace model. Let
	\[
		F(\bz_1,\ldots,\bz_d)
		:=
		\left\{
		\bigl(\bu,\sqrt{1-\|\bu\|^2}\bigr):
		\bu\in[\bz_1,\ldots,\bz_d]
		\right\}
	\]
	be the hyperbolic $(d-1)$-simplex in $\Sigma$ with vertices
	$\bp_1,\ldots,\bp_d$, and let
	$S(\bz_1,\ldots,\bz_d)$ be the hyperbolic $d$-simplex with facet
	$F(\bz_1,\ldots,\bz_d)$ and additional ideal vertex $\infty$, that is,
	\[
		S(\bz_1,\ldots,\bz_d)
		=
		\left\{
		(\bu,y)\in\HH^d:
		\bu\in[\bz_1,\ldots,\bz_d],\
		y\ge\sqrt{1-\|\bu\|^2}
		\right\}.
	\]
	Using the hyperbolic volume representation
	\eqref{eq:HyperbolicVolume}, we obtain
	\[
		\vol_{\HH^d}\bigl(S(\bz_1,\ldots,\bz_d)\bigr)
		=
		\frac{1}{d-1}
		\int_{[\bz_1,\ldots,\bz_d]}
		(1-\|\bu\|^2)^{-\frac{d-1}{2}}\,\dint\bu,
	\]
	while the computation of the hyperbolic area element in the proof of
	Theorem~\ref{thm:LocalVolume}, applied with $\bw=o$ and $t=1$,
	yields
	\[
		\cH_{\HH^d}^{d-1}\bigl(F(\bz_1,\ldots,\bz_d)\bigr)
		=
		\int_{[\bz_1,\ldots,\bz_d]}
		(1-\|\bu\|^2)^{-d/2}\,\dint\bu.
	\]
	Consequently, for $\bz_i=\sqrt{1-ty_i}\,\bu_i$ \eqref{eq:Auxiliary}
	equals
	\[
		\begin{cases}
			(d-1)\vol_{\HH^d}\bigl(S(\bz_1,\ldots,\bz_d)\bigr),
			 & \ell=0, \\[1mm]
			\cH_{\HH^d}^{d-1}\bigl(F(\bz_1,\ldots,\bz_d)\bigr),
			 & \ell=1.
		\end{cases}
	\]
	As $t\downarrow0$, the vertices $\bz_i=\sqrt{1-ty_i}\,\bu_i$
	converge to $\bu_i\in\SS^{d-2}$, the simplex
	$S(\bz_1,\ldots,\bz_d)$ converges to the ideal simplex
	$S^\infty(\bu_1,\ldots,\bu_d)$, and the facet
	$F(\bz_1,\ldots,\bz_d)$ converges to the ideal
	$(d-1)$-simplex $F^\infty(\bu_1,\ldots,\bu_d)$ opposite the
	vertex $\infty$. Hence, for almost every
	$(\bu_1,\ldots,\bu_d)$, this factor converges as
	$t\downarrow0$ to $(d-1)^{1-\ell}G_{d,\ell}(\bu_1,\ldots,\bu_d)$, where we recall the definition \eqref{eq:Gdl} of $G_{d,\ell}(\bu_1,\ldots,\bu_d)$. For $\ell=0$, the latter quantity is controlled by the maximal volume
	of a hyperbolic ideal $d$-simplex. For $\ell=1$ and $d\ge3$, it is
	controlled by the maximal volume of a hyperbolic ideal
	$(d-1)$-simplex. Thus the same truncation and dominated-convergence
	argument as in the proof of
	Theorem~\ref{thm:AsymptoticsGeneralD} yields
	\begin{equation}\label{eq:JdellAsymptotic}
		J_{d,\ell}(1-t)
		=
		C_{d,\ell}\,
		t^{-\frac{d(d-1)}2}
		\bigl(1+R_{d,\ell}(t)\bigr),
		\qquad t\in(0,1),
	\end{equation}
	where \(R_{d,\ell}(t)\to0\) as \(t\downarrow0\). Moreover,
	\(R_{d,\ell}\) is bounded on \((0,1)\). The above asymptotic gives
	boundedness of $R_{d,\ell}$ in a neighbourhood of zero, while away from zero this follows
	from the finiteness and monotonicity of \(J_{d,\ell}\). Namely, for every
	fixed \(t_0\in(0,1)\), we have $J_{d,\ell}(1-t)\le J_{d,\ell}(1-t_0)$, $t\in[t_0,1)$.

	To identify the constant in \eqref{eq:JdellAsymptotic}, note that the radial
	integrations contribute
	\[
		2^{-d}
		\left(
		\int_1^\infty y^{-\frac{d+1}{2}}\,\dint y
		\right)^d
		=
		\frac{1}{(d-1)^d},
	\]
	which, together with the factor $(d-1)^{1-\ell}$, gives precisely
	the normalization in the definition of $C_{d,\ell}$.

	We now make the global substitution $s=1-\gamma^{2/(d-1)}x$ with $x\in(0,\gamma^{-2/(d-1)})$.
	Using \eqref{eq:JdellAsymptotic}, all powers of $\gamma$ cancel, and
	the integrand in \eqref{eq:LocalFunctionalAsymptotics} converges
	pointwise to
	\[
		\frac{C_{d,\ell}}{2d}\,
		x^{-1-\frac{(d-1)^2}{2}}
		\exp\!\left(
		-\frac{\kappa_{d-1}}{d-1}
		x^{-\frac{d-1}{2}}
		\right).
	\]
	The boundedness of \(R_{d,\ell}\), together with the estimates for
	\(A\) used in the proof of Theorem~\ref{thm:AsymptoticsGeneralD},
	allows us to apply the same dominated-convergence argument.
	Hence
	\begin{align*}
		 & \lim_{\gamma\downarrow0}
		\frac{1}{\gamma}\,
		\EE F_{\gamma,\lambda}^{(\ell)}(W)
		=
		\frac{C_{d,\ell}}{2d}
		\int_0^\infty
		x^{-1-\frac{(d-1)^2}{2}}
		\exp\!\left(
		-\frac{\kappa_{d-1}}{d-1}
		x^{-\frac{d-1}{2}}
		\right)
		\,\dint x.
	\end{align*}
	With the substitution $u=x^{-(d-1)/2}$, the remaining integral
	equals $2\,(d-2)!\,(d-1)^{d-2}\,\kappa_{d-1}^{-(d-1)}$.
	Consequently,
	\[
		\lim_{\gamma\downarrow0}
		\frac{1}{\gamma}\,
		\EE F_{\gamma,\lambda}^{(\ell)}(W)
		=
		C_{d,\ell}\,
		\frac{(d-2)!(d-1)^{d-2}}{d\,\kappa_{d-1}^{d-1}}.
	\]
	For $\ell=0$, this yields the local-volume asymptotic, while for
	$\ell=1$ it gives the local surface-area asymptotic for $d\ge3$.
	This yields the asserted
	asymptotics for the normalised ratios
	$V_{\gamma,\lambda}(W)$ and
	$S_{\gamma,\lambda}(W)$.

	It remains to consider the local surface area in dimension $d=2$. In
	this case, a direct one-dimensional computation gives
	\begin{equation}\label{eq:J21Asymptotic}
		J_{2,1}(1-t)
		=
		4t^{-1}\log\frac{1}{t}
		+
		O(t^{-1}),
		\qquad t\downarrow0.
	\end{equation}
	Indeed, with $z_i=\frac{u_i}{\sqrt{1+u_i^2}}$ and $U=\sqrt{\frac{1-t}{t}}$,
	the integral $J_{2,1}(1-t)$ becomes
	\[
		\int_{[-U,U]^2}
		\left|
		\frac{u_1}{\sqrt{1+u_1^2}}
		-
		\frac{u_2}{\sqrt{1+u_2^2}}
		\right|
		\left|
		\operatorname{arsinh}(u_1)
		-
		\operatorname{arsinh}(u_2)
		\right|
		\,\dint u_1\dint u_2.
	\]
	The two quadrants in which $u_1$ and $u_2$ have opposite signs
	contribute
	$8U^2\log U+O(U^2)$, whereas the remaining two quadrants contribute
	only $O(U^2)$ as $U\to\infty$. Since $U^2\sim t^{-1}$ as $t\to 0$, this proves
	\eqref{eq:J21Asymptotic}.

	Substituting \eqref{eq:J21Asymptotic} into
	\eqref{eq:LocalFunctionalAsymptotics} and using
	$t=\gamma^2x$, together with
	\[
		A(1-t)=\frac{2}{3}(1-t)^{3/2}{}_2F_1\!\Big(\tfrac32,\tfrac32;\tfrac52;1-t\Big)
		=2\sqrt{\frac{1-t}{t}}-\pi+2\arcsin(\sqrt{t})
		=
		2t^{-1/2}-\pi+O(t^{1/2}),
		\quad t\downarrow0,
	\]
	gives
	\[
		\EE S_{\gamma,\lambda}(W)
		=
		2\gamma\log\frac{1}{\gamma}
		+
		O(\gamma).
	\]

	Finally, the Efron-type identity gives
	\[
		\xi_0(\lambda^{d-1},\lambda)
		=
		\frac{1}{d-1}
		-
		\lambda^{d-1}
		\EE\vol_{\HH^d}\bigl(
		K_{\lambda^{d-1},\lambda}
		\cap C_\lambda([0,1]^{d-1})
		\bigr).
	\]
	Applying the local-volume asymptotic proved above, we obtain
	\[
		\xi_0(\lambda^{d-1},\lambda)
		=
		\frac{1}{d-1}
		-
		C_{d,0}\,
		\frac{(d-2)!(d-1)^{d-3}}{d\,\kappa_{d-1}^{d-1}}\,
		\lambda^{d-1}
		+
		o(\lambda^{d-1}),
	\]
	which completes the proof.
\end{proof}

\begin{proof}[Proof of Proposition~\ref{cor:Vold=2d=3}]
	We need to evaluate the integral
	$$
		J := \int_{(\SS^1)^3}\Delta_2(\bu_1,\bu_2,\bu_3)\vol_{\HH^3}\big(S^\infty(\bu_1,\bu_2,\bu_3)\big)\,\sigma(\dint\bu_1)\,\sigma(\dint\bu_2)\,\sigma(\dint\bu_3).
	$$
	The volume of the $3$-dimensional ideal hyperbolic tetrahedron $S^\infty(\bu_1,\bu_2,\bu_3)$ can be expressed in terms of the Lobachevsky function $\russianL(x):=-\int_0^x \log|2\sin t|\,\dint t$, $x>0$. If we parametrize $\bu_1$, $\bu_2$ and $\bu_3$ as
	$$
		\bu_1=e^{i\theta},\qquad\bu_2=e^{i(\theta+\alpha)},\qquad\bu_3=e^{i(\theta+\alpha+\beta)},\qquad \alpha,\beta>0,\;\gamma:=2\pi-\alpha-\beta>0,
	$$
	then
	$$
		\Delta_2(\bu_1,\bu_2,\bu_3)=\frac{1}{2}(\sin\alpha+\sin\beta+\sin\gamma),\qquad\vol_{\HH^3}\big(S^\infty(\bu_1,\bu_2,\bu_3)\big)
		=
		\russianL\Big(\frac{\alpha}{2}\Big)
		+
		\russianL\Big(\frac{\beta}{2}\Big)
		+
		\russianL\Big(\frac{\gamma}{2}\Big),
	$$
	independently of $\theta$, see \cite[Theorem 10.4.10]{Ratcliffe} or \cite[Proposition 36.5.13]{Voight}.
	We can therefore use rotational invariance and separate the two possible cyclic orders of the three labelled points. This yields
	\[
		J= 2\pi
		\int_D
		(\sin\alpha+\sin\beta+\sin\gamma)
		\Big(
		\russianL\Big(\frac{\alpha}{2}\Big)
		+
		\russianL\Big(\frac{\beta}{2}\Big)
		+
		\russianL\Big(\frac{\gamma}{2}\Big)
		\Big)
		\,\dint\alpha\,\dint\beta,
	\]
	where $D:=\{(\alpha,\beta):\alpha>0,\ \beta>0,\ \alpha+\beta<2\pi\}$.
	By symmetry of the integrand,
	\[
		\begin{aligned}
			J
			 & =
			6\pi
			\int_0^{2\pi}
			\russianL\!\left(\frac{\alpha}{2}\right)
			\left(
			\int_0^{2\pi-\alpha}
			\sin\alpha+\sin\beta+\sin(2\pi-\alpha-\beta)
			\,\dint\beta
			\right)\dint\alpha .
		\end{aligned}
	\]
	The inner integral equals $(2\pi-\alpha)\sin\alpha+2(1-\cos\alpha)$.
	Thus, after the change of variables $x=\alpha/2$,
	\[
		J
		=
		24\pi
		\int_0^\pi
		\russianL(x)\bigl((\pi-x)\sin(2x)+1-\cos(2x)\bigr)
		\,\dint x .
	\]
	Since $\russianL(\pi-x)=-\russianL(x)$, whereas $1-\cos(2x)$ is symmetric
	with respect to $x=\pi/2$, the integral $\int_0^\pi\russianL(x)(1-\cos(2x))\,\dint x$ vanishes. Hence
	\[
		J
		=
		24\pi
		\int_0^\pi
		(\pi-x)\sin(2x)\russianL(x)\,\dint x .
	\]
	Using $\russianL(x)=\frac12\operatorname{Cl}_2(2x)$, with the Clausen
	function of order $2$, and putting $\theta=2x$, we obtain
	\[
		J
		=
		3\pi
		\int_0^{2\pi}
		(2\pi-\theta)\sin\theta\,\operatorname{Cl}_2(\theta)
		\,\dint\theta .
	\]
	Set $f(\theta):=\sin\theta\,\operatorname{Cl}_2(\theta)$.
	Since $f(2\pi-\theta)=f(\theta)$, we have
	\[
		\begin{aligned}
			\int_0^{2\pi}(2\pi-\theta)f(\theta)\,\dint\theta
			 & =
			\int_0^\pi (2\pi-\theta)f(\theta)\,\dint\theta
			+
			\int_\pi^{2\pi} (2\pi-\theta)f(\theta)\,\dint\theta \\
			 & =
			\int_0^\pi (2\pi-\theta)f(\theta)\,\dint\theta
			+
			\int_0^\pi \theta f(\theta)\,\dint\theta            \\
			 & =
			2\pi\int_0^\pi f(\theta)\,\dint\theta .
		\end{aligned}
	\]
	Thus,
	\[
		J
		=
		6\pi^2
		\int_0^\pi
		\sin\theta\,\operatorname{Cl}_2(\theta)\,\dint\theta .
	\]
	The last integral equals $\pi/2$ by \cite[Theorem 6]{Lima}.
	Consequently, $J=6\pi^2\cdot\frac{\pi}{2}=3\pi^3$.
	Finally, since the definition of $C_{3,0}$ contains the
	prefactor $(d-1)^{-(d-1)}=2^{-2}$, $C_{3,0}=\frac{1}{4}\cdot J=\frac{3\pi^3}{4}$.

	It remains to compute the constant $C_{3,1}$. In this case,
	$F^\infty(\bu_1,\bu_2,\bu_3)$ is an ideal triangle in $\HH^2$, whose hyperbolic area is just $\pi$.
	Therefore, using the spherical-simplex integral already evaluated in
	\eqref{eq:Cd}, we obtain
	\begin{align*}
		C_{3,1}
		 & =
		\frac{1}{2^3}
		\int_{(\SS^1)^3}
		\Delta_2(\bu_1,\bu_2,\bu_3)\,
		\cH_{\HH^3}^{2}\bigl(
		F^\infty(\bu_1,\bu_2,\bu_3)
		\bigr)\,
		\sigma(\dint\bu_1)\sigma(\dint\bu_2)\sigma(\dint\bu_3) \\
		 & =
		\frac{\pi}{8}
		\int_{(\SS^1)^3}
		\Delta_2(\bu_1,\bu_2,\bu_3)\,
		\sigma(\dint\bu_1)\sigma(\dint\bu_2)\sigma(\dint\bu_3)
		=
		\frac{\pi}{8}\cdot12\pi^2
		=
		\frac{3\pi^3}{2}.
	\end{align*}
	This completes the proof.
\end{proof}

\subsection*{Acknowledgement}
This work was initiated during the Mini-Workshop \textit{Hyperbolic Meets Stochastic Geometry} at the Mathematisches Forschungsinstitut Oberwolfach (MFO), see \cite{KT26OWR}. AG and CT were supported by the DFG priority program SPP 2265 \textit{Random Geometric Systems}. AG was supported by the DFG under Germany's Excellence Strategy  EXC 2044 -- 390685587, \textit{Mathematics M\"unster: Dynamics - Geometry - Structure}.

During the preparation of this work, the authors used Claude (Anthropic)
and ChatGPT (OpenAI) as tools for technical and editorial assistance,
including the generation of code, in particular to run Monte Carlo checks.
All mathematical statements and proofs were independently verified and finalised by the authors.

\end{document}